\documentclass[preprint,3p,times]{elsarticle}
\usepackage{bbm,bm,url}

\usepackage{subfigure}
\usepackage{xcolor}
\usepackage[math]{easyeqn}

\usepackage[normalem]{ulem}

\usepackage[ruled]{algorithm2e}
\renewcommand{\algorithmcfname}{ALGORITHM}
\SetAlFnt{\small}
\SetAlCapFnt{\small}
\SetAlCapNameFnt{\small}
\SetAlCapHSkip{0pt}
\IncMargin{-\parindent}

\usepackage{multicol}

\usepackage{graphicx}
\graphicspath{{figures/}}

\usepackage{mfm_macros}

\newenvironment{cases}{\left\{\begin{array}{ll}}{\end{array}\right.}

\newenvironment{pmatrix}{\left(\begin{array}{ccccccc}}{\end{array}\right)}

\newcommand{\Afun}{\mathop{\!\,\mathcal{A}}}
\newcommand{\Efun}{\mathop{\!\,\mathcal{E}}}
\newcommand{\Ffun}{\mathop{\!\,\mathcal{F}}}
\newcommand{\Gfun}{\mathop{\!\,\mathcal{G}}}
\newcommand{\Lfun}{\mathop{\!\,\mathcal{L}}}
\newcommand{\Sfun}{\mathop{\!\,\mathcal{S}}}
\newcommand{\Qfun}{\mathop{\!\,\mathcal{Q}}}
\newcommand{\argmin}{\mathop{\mathrm{arg\,min}}}

\newcommand{\sinc}{\mathop{\mathrm{sinc}}}
\newcommand{\vinf}{v_{\infty}}
\newcommand{\vzero}{v_{\circ}}
\newcommand{\tzero}{t_{\circ}}
\newcommand{\tmin}{t_{\min}}
\newcommand{\tmax}{t_{\max}}
\newcommand{\tinf}{t_{\infty}}

\newcommand{\tguess}{t_{\mathrm{guess}}}
\newcommand{\tsw}{t_{\sigma}}

\newcommand{\troot}{t_{\star}}

\newcommand{\ssw}{s_{\sigma}}

\newcommand{\tvzero}{t_{\vzero}}
\newcommand{\tvinf}{t_{\vinf}}

\newcommand{\svzero}{s_{\vzero}}

\newcommand{\szero}{s_{\circ}}
\newcommand{\smin}{s_{\min}}
\newcommand{\smax}{s_{\max}}

\newtheorem{theorem}{Theorem}
\newtheorem{lemma}[theorem]{Lemma}
\newtheorem{remark}{Remark}
\newtheorem{proposition}[theorem]{Proposition}
\newtheorem{corollary}[theorem]{Corollary}

\newproof{pf}{Proof}
\newenvironment{proof}{\begin{pf}}{\end{pf}}

\begin{document}

\begin{frontmatter}

\title{Semi-Analytical Minimum Time Solution for the Optimal Control\\ of a Vehicle subject to Limited Acceleration}
\author[EB]{Enrico Bertolazzi}\ead{enrico.bertolazzi@unitn.it}
\author[MF]{Marco Frego}
\address[EB]{DII - Dipartimento di Ingegneria Industriale}
\address[MF]{DISI - Dipartimento di Ingegneria e Scienza dell'Informazione}


\begin{abstract}
A semi-analytic solution of the minimum time optimal 
control problem of a car-like vehicle is herein presented.  The vehicle is subject 
to the effect of laminar (linear) 
and aerodynamic (quadratic) drag, taking into account the asymmetric
bounded longitudinal accelerations. This yields a nonlinear differential equation
of Riccati kind. 
The associated analytic solution exists in closed form, but is 
numerically ill conditioned in some
situations, hence it is reformulated using asymptotic expansions
that keep the numerical computations accurate and robust in all cases.

Therefore the proposed algorithm is designed to be fast and robust 
in sight to be the fundamental module of a more extended optimal 
control problem that considers a given clothoid curve as the trajectory and 
the presence of a constraint on the lateral acceleration of the vehicle. 
The numerical stability of the computation for limit values of the parameters is essential as 
showed in the numerical tests.
\end{abstract}


\end{frontmatter}

\section{Introduction and Motivation}
\label{sec:intro}

An important objective in the field of mobile robotics since the first papers 
of Dubins~\cite{dubins1957curves} is to find the optimal trajectory between two points with specified tangents
vectors and velocities. Moreover, other non dynamic constraints~\cite{Bertolazzi:2006kc,Gerdts:2013,rizano2013global} are considered, such 
as the geometric continuity of the path~\cite{Rao:2014b,Diehl:2009}, the restrictions on the maximum curvature~\cite{DaLio2014,Frazzoli:2014}, the avoidance 
of obstacles~\cite{Amditis:2010fx,Bertolazzi:2010ud,Gerdts:2013b}. There are mainly two global views of this problem: the first 
tries to solve in one shot the generation of the path and the optimal law to travel it, the 
second splits the problem into two separate parts: path generation and optimisation of the dynamics. 
The first approach is done essentially numerically, considering very complex vehicle models,
geometric constraints, and various other characteristics which approximate the real world 
in a very satisfactory way. The numeric solution obtained in this way is a formidable task by itself
and it requires a deep knowledge of the physical problem and an efficient optimal control 
solver~\cite{Bertolazzi:2006kc,Buskens:2013aa,Diehl:2011,Rao:2014,ICLOCS:2010}.
These solvers transform the problem into a Nonlinear Programming
Problem (NLP) or into a Boundary Value Problem (BVP) that is efficiently solved by numerical
methods~\cite{Ascher:1995,Boyd:2014,Mazzia:2002,mazzia:2011,nocedal:2006,Biegler:2005}.

It is not possible to solve those problems analytically (and thus precisely and quickly)
because they are too complex or simply they do not possess a closed form solution. 
Even when the analytic or semi-analytic solution exists~\cite{Velenis:2005aa,Velenis:2008},
it can be numerically ill conditioned or of impractical use, as in the present case.
The purely numerical solutions are in general very good and 
accurate, but at the price of a very high computational cost, which implies that these
methods are not suitable for real time computation. 
Thus existing numerical solver cannot be used in online simulations,
where a quick feedback, e.g. in the case of an unexpected situation or a sudden event,
must be managed.
\\
On the other hand, splitting the problem in two sub-problems allows us to consider them 
separately, with the reasonable expectation that they become easier to tackle and admit an analytical solution. 
This is not usually the case,
unless the problem is simplified via assumptions on the hypotheses and via linearisation of the 
differential equations of the dynamics. The advantage of such an approximate but analytic solution 
is that it is of quick evaluation and therefore suitable for real time applications. Complete
analytic solutions are seldom obtained and almost only on toy problems of limited application, 
however, semi-analytical solutions are often a good trade-off between low computational times and 
usefulness. \\
With this second framework in mind, in~\cite{ECC:2016} we solved the problem of a car-like vehicle 
travelling on a given clothoid, with assigned initial and final position, tangent and velocity.
This results in an optimal control problem with a Riccati ODE in the dynamical system.
The curve is traversed in minimum time, controlling the limited longitudinal acceleration and 
considering the bound on the maximum lateral acceleration of the vehicle, a quantity that 
connects the curvature of the path with the velocity of the vehicle,~\cite{ECC:2016}. 
The complete problem is made up of several blocks: the algorithm for solving the $G^1$ Hermite
Interpolation Problem with a clothoid~\cite{Bertolazzi:2015aa}, necessary to compute the path;
the description of the car-like model and its simplification in order to find analytical useful
solutions~\cite{ECC:2016}, 
the development of a fast and reliable tool that solves the ODEs of the basic optimal control
(which is the aim of the present work), and
a solution method for the optimal control problem with the constraint on the lateral acceleration.
All these blocks are used inside a high level optimiser that constructs the trajectory of the car-like
vehicle along a sequence of given points~\cite{FregoBBFP16Automatica,FregoBBFP16CDC,rizano2013global}.

In this work, we focus on the fast and accurate numerical computation of
the analytic solution of an optimal control problem arising from a Riccati dynamical system.
We remark, since  the limit of the lateral acceleration is a function of the state only
(and not of the control variable),  
that we can solve the complete problem in two phases.
In the first phase, we solve the OCP accounting for
the longitudinal constraint only. 
The problem turns out
to be classically Bang-Bang, yielding a switched hybrid system. We obtain analytic expressions for
the optimal state variables as well as robust numerical routines to evaluate them in all cases, e.g.,
when the input parameters rise computational instabilities.  Moreover, the various combinations
of possible initial and final conditions produce many different analytic solution that are worth
being studied independently because it is possible to synthesize them with only two stable expressions, even when 
the parameters of the problem lead to numerical instabilities.
In the second phase, which is not scope of this work, we combine the 
obtained solutions together with the bound on the lateral acceleration in order to synthesize the optimal control 
as it is sketched in~\cite{ECC:2016}. \\ 

In Section~\ref{sec:ocp} we present the Optimal Control Problem (OCP) with the description
of the Riccati differential equation formulated with the time as the independent variable,
we show the Bang-Bang nature of the solution and how to find the single switching point and
the final total time of the maneuver, assuming that the analytic solutions of the ODEs 
of problem are available. 
The derivation of such analytic solutions is devoted to 
Section~\ref{sec:solution}, where we give also the definition 
domains of the velocty  $v(t)$ and space $s(t)$.
The aim of Section~\ref{sec:stable} is to provide numerically stable versions of the solutions
previously obtained. In Section~\ref{sec:v_of_s} we introduce the change of variable from 
time $t$ to space $s$ for the velocity. In Section~\ref{sec:ocp:s} we reformulate and solve
the OCP presented in Section~\ref{sec:ocp} but reformulated with the space as independent variable.
Its solution is one nonlinear equation that we solve with the application of Newton's method.
Section~\ref{sec:tests} shows the practical implementation of the stable equations presented 
in four different test cases for different variations of the parameters.
The last section are the conclusions with some insights in application of the present work
as a building block for more complex computations and OCP problems.

\section{The Optimal Control Problem}\label{sec:ocp}
The minimum time optimal control problem of a car-like vehicle that follows any smooth path, 
neglecting for the moment any constraints on the lateral acceleration, is~\cite{ECC:2016}:

Find $a(t)\in[-\ab,\ap]$ that minimizes the total time $T$ subject to:\\
\begin{EQ}[rcl]\label{OCP}
   s'(t) &=& v(t), \\
   v'(t) &=& \acc(t)-c_0v(t)-c_1v^2(t),\qquad  -\ab\leq a(t) \leq \ap,\\
   s(0) &=& 0,\qquad s(T) = L, \qquad v(0) = v_i,\qquad v(T) = v_f,
\end{EQ}
where $T$ is the final time to be optimised, $s(t)$ represents the space variable in 
curvilinear coordinates, $v(t)$ is the velocity of the vehicle, $\acc(t)$ is
the controlled acceleration asymmetrically bounded for positive $\ab$ and $\ap$, the laminar friction
is given by the nonnegative coefficient $c_0$, the aerodynamic drag is modelled by the nonnegative
coefficient $c_1$. All the
boundary conditions are fixed: without loss of generality the initial position is $0$,
the final position is $L$ (the length of the curve), the initial and final velocities are
$v_i$ and $v_f$, respectively, both assumed nonnegative.\\
The solution of the OCP can be obtained via general optimal control software, as 
described~\cite{Bertolazzi:2006kc,Betts:2010,Betts:2007,Biral:2014aa,Campbell:2016,Diehl:2011,Rao:2014}, 
where various methods for OCP are presented. However they cannot take advantage of the 
analytic solution and hence, although being more general purpose, they are slower than 
the present method when applied to this specific problem.\\ 

We proceed by showing the nature of the optimal control using the Maximum Principle of
Pontryagin. The Hamiltonian of the minimum time optimal control problem~\eqref{OCP} is
\begin{EQ}
  \mathcal{H}(s,v,\lambda_1,\lambda_2,\acc) = 1  +\lambda_1v + \lambda_2 ( \acc-c_0v-c_1v^2)
\end{EQ}
and the control $\acc$ appears linearly. Its optimal synthesis is obtained from Pontryagin's Maximum (minimum) Principle (PMP), and is
\begin{EQ}\label{ocp}
\acc(t) = \argmin_{\acc\in[-\ab,\ap]} \mathcal{H}(s(t),v(t),\lambda_1(t),\lambda_2(t),\acc) = 
  \begin{cases}
    \ap  & \textrm{if $\lambda_2(t)<0$},\\
    -\ab & \textrm{if $\lambda_2(t)>0$},\\
    \acc_{\mathrm{sing}} & \textrm{if $\lambda_2(t)=0$}.
  \end{cases}
\end{EQ}
Hence the solution of the PMP produces a typical Bang-Bang controller.
The term $\acc_{\mathrm{sing}}$ represents
a possible singular control when $\lambda_2$ is identically zero on an interval.
  The equations for the adjoint variables are
 \begin{EQ}
   \lambda_1' = 0, \qquad
   \lambda_2' = \lambda_1 (c_0+2c_1v)-\lambda_1.
\end{EQ}
As the state variables have complete boundary conditions, no initial or final conditions are required for the two multipliers. The singular control $a_{\mathrm{sing}}$ is derived by using the above expressions for the multipliers, but according to a well known result in Optimal Control Theory~\cite{Athans:1966},
the solution of problem 
\eqref{OCP}, if it exists, has at most one switching
instant (denoted as $\tsw$) when the control value changes.  
\begin{lemma}
   The optimal control for problem \eqref{OCP}, if exists, has at most one switching 
   point and the control law is Bang-Bang.
\end{lemma}
\begin{proof}
  See \cite{Athans:1966}. The idea is that the multiplier $\lambda_2$, in absence of bounds on the state variables, is strictly monotone increasing, hence there is at most one isolated zero and no singular arcs can exist.\qed
\end{proof}

This implies that the 
optimal control has to be chosen from a family of four candidate controls.  
The first and second correspond to
pure acceleration (i.e., $\acc(t) \equiv \ap$) or braking manoeuvre (i.e., $\acc(t) \equiv -\ab$), the third and fourth are a
combination of the two.
 Having explicit
expressions for the optimal states, with the complete boundary conditions,
the solution of each case is obtained by the
solution of a nonlinear system in two unknowns: the switching
time $\tsw$ and the final time $T$. It is possible to rule out the case of a
braking manoeuvre followed by an acceleration because it is not
optimal. Moreover, the cases of pure acceleration or braking are very
unlikely, because the boundary conditions should exactly satisfy the
boundary value problem \eqref{OCP}. Hence, only the case of
acceleration and braking with a switching time (possibly degenerate zero length 
interval at the extrema of the time interval)
is herein considered. It follows that the solution is given by the
intersection at the switching point $\tsw$ of the two curves of the
velocity and the space $v(t)$ and $s(t)$. The closed form solution of
$v(t)$ and $s(t)$ and its accurate computation is a difficult and important part, therefore it is
postponed to Section \ref{sec:solution}. \\
We are here interested to find the existence conditions for the optimal control of problem 
\eqref{OCP}. They must depend on the assigned initial and final positions and velocities. 
We have to ensure that there is enough acceleration (respectively deceleration) 
together with the initial and final velocities to travel the curve. Those values should be 
compatible with the extremal manoeuvres of pure acceleration and pure braking. 

If the final 
velocity $v_f$ is greater than the reference velocity $v_f^{\max}$ obtained starting 
from $v_i$ with maximum acceleration, then the problem does not have solution. 
Reference velocity $v_f^{\min}$ is the final velocity obtained starting 
from $v_i$ with maximum deceleration and $v_f^{\min}=0$ if the maximum deceleration
produces a solution with $v(t)=0$ for $s(t) < L$, i.e., the
vehicle is stopped before to reach the final distance $L$.
If the final velocity $v_f$ is lower than the reference 
velocity $v_f^{\min}$ the problem is not solvable.

In the other cases a solution exists. Those limits values are discussed in the next sections, together with the analytic expression of the integration of the differential equations.

The complete solution of the Optimal Control Problem~\ref{OCP} is reduced
to finding the optimal switching instant $\tsw$ and the final minimum time $T$. This is done equating the arcs of positive acceleration of velocity and space with the corresponding arcs of negative acceleration. It results a system of two nonlinear equations in the unknowns $\tsw$ and $T$:
\begin{EQ}\label{eq:sys_bb}
  v_L(\tsw) = v_R(\tsw-T),\qquad
  s_L(\tsw) = s_R(\tsw-T). 
\end{EQ}
Where $s_L(t)$ and $v_L(t)$ are the solutions of the ODE 
\begin{EQ}
   \left\{\begin{array}{r@{\,}c@{\,}lr@{\,}c@{\,}l}
   s'(t) &=& v(t),\quad & s(0)&=&0 \\
   v'(t) &=& \ap-c_0v(t)-c_1v^2(t),\quad & v(0)&=&v_i \\
   \end{array}\right.   
\end{EQ}
whereas $s_R(t)$ and $v_R(t)$ are the solutions of the ODE 
\begin{EQ}
   \left\{\begin{array}{r@{\,}c@{\,}lr@{\,}c@{\,}l}
   s'(t) &=& v(t),\quad & s(T)&=&L \\
   v'(t) &=& -\ab-c_0v(t)-c_1v^2(t),\quad & v(T)&=&v_f \\
   \end{array}\right.   
\end{EQ}
The next parts of the present work are devoted to the computation of the analytic
expressions of $s(t)$, $v(t)$ and to the study their numeric stable implementation.
The nonlinear system~\eqref{eq:sys_bb} can be further reduced to the computation
of the zero of a single nonlinear equation. From this reduction it is possible to show
that there is a unique solution and that the Newton method always converges. However this
reduction is possible at the price of an involved change of variable, hence we prefer to 
solve directly system~\eqref{eq:sys_bb} which is well behaved with a small computational cost.
\section{Formal Analytic Solution of Riccati ODE}\label{sec:solution}
From the analysis of the optimal control problem in the previous section, we conclude that the second differential equation of \eqref{OCP}, i.e.,
\begin{EQ}\label{eq:riccati}
  v'(t)=\acc(t)-c_0v(t)-c_1v(t)^2, \quad v(0)=\vzero\geq 0,
\end{EQ}
should be considered only for piecewise constant values of the control $\acc(t)$. This simplifies a lot the integration of the ODE, which is a differential equation of Riccati that can be solved only for a small class of functions $\acc(t)$, \cite{abramowitz:1964,nist:2010,zaitsev:2002}.
The constant functions are in this class and allow to solve the ODE by means of separation of variables or transforming the Riccati into a Bernoulli and then with standard techniques into a linear ODE. The shape of the solution varies according to the value of the initial condition, see Figure \ref{fig:cases:v} for a graphic reference.

For constant positive accelerations $\acc>0$ and positive $c_1>0$, the threshold is given by the asymptotic value $\vinf$.
For initial conditions $\vzero<\vinf$, the velocity is monotone increasing towards  the asymptote $\vinf$, 
for $\vzero>\vinf$ the velocity is monotone decreasing towards the asymptotic value $\vinf$, lastly,
if $\vzero=\vinf$ the velocity is constant.

For negative accelerations $\acc<0$, the situation is slightly different, and the sign of the quantity
$c_0^2+4\acc c_1$ must be considered: this leads to two cases. 
Also the cases of $\acc=0$, $c_0\approx 0$ and/or $c_1\approx 0$ should be analysed, because they originate
 numerical difficulties when evaluating the analytic expression of $v(t)$. Once the explicit expression for $v(t)$ is given, the integration of the ODE for the space variable $s'=v$ is possible in closed form. We have collected the various events in different cases, that we will discuss next.\\
\begin{figure}[!tb]
  \begin{center}
    \begin{tabular}{cc}
      \subfigure[Velocity plot for $\acc\geq0$. Cases (a), (b) and (c). Case (a) is the horizontal  
                 line of constant velocity $v(t)=v_{\infty}$. 
                 The lower curve represents case (b), when $\vzero$ is lower than
                 the asymptotic velocity $\vinf$; case (c) is the upper curve, when 
                 $\vzero$ is higher than $\vinf$. ]
      {\label{fig:cases:v:a}\includegraphics[scale=0.9]{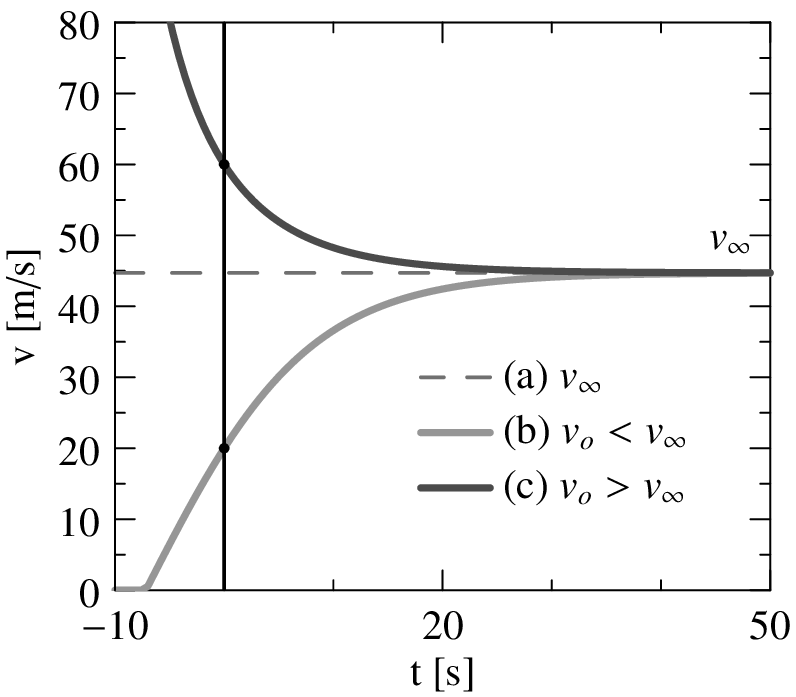}}
      &
      \subfigure[Velocity plot for $\acc<0$.  Cases (d) and (e). Arcs indicate negative constant acceleration $\acc=-\ab$. 
                 There are two cases according to $w$. When $w$ is real, we fall in case (d), when $w$ is complex 
                 we fall in case (e).]
      {\label{fig:cases:v:b}\includegraphics[scale=0.9]{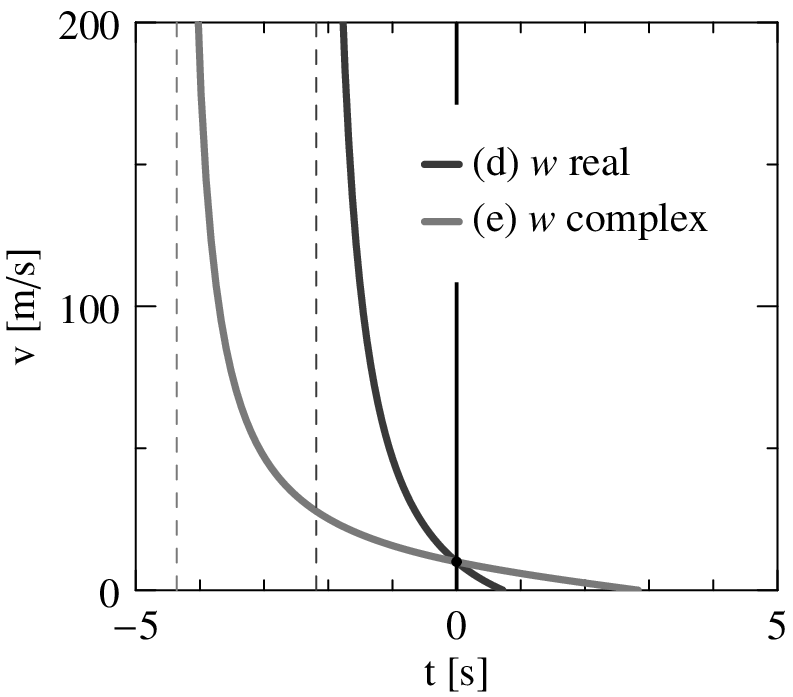}}%
      \end{tabular}
   \end{center}
   \vspace{-1em}
   \caption{The five cases of velocity and travelled space as listed in Lemma \ref{lem:velocity}. In all cases, the acceleration
                 is constant and positive. The green arcs indicate positive constant control $\acc=\ap$, red arcs have negative constant acceleration $\acc=-\ab$.}
   \label{fig:cases:v}
   \end{figure}
A straightforward integration of \eqref{eq:riccati} with a general initial condition $\vzero$ yields a long expression, which involves inverse trigonometric functions and their hyperbolic analogue according to the sign of the acceleration, the integration of the velocity has those functions nested inside a logarithm, making $v(t)$ and $s(t)$ difficult to accurately evaluate. Thus, in some of the cases presented next, the solution of equation~\eqref{eq:riccati} is manipulated in order to use a more compact and stable formula with some auxiliary constants that are here introduced:
\begin{EQ}[rclrclrcl]\label{eq:constants}
          w     &\DEF & \sqrt{ c_0^2+4\acc c_1 }, \quad &
     \alpha     &\DEF & \dfrac{w+c_0}{2}, \quad &
     \beta      &\DEF & \dfrac{w-c_0}{2}, \\
     \gamma     &\DEF & c_1\vzero+\alpha,\quad &
     \vinf      &\DEF & \dfrac{\beta}{c_1}=\dfrac{\acc}{\alpha}, \quad &
     \acczero   &\DEF & (c_1\vzero+c_0)\vzero \\
     \theta     &\DEF& \arctan\dfrac{\abs{w}}{c_0},\quad&
     \theta_0   &\DEF& \arctan\dfrac{\vzero\abs{w}}{\vzero c_0+2\abs{\acc}},\quad &
     \theta_1   &\DEF& \arctan\dfrac{\abs{w}}{2c_1\vzero+c_0}
\end{EQ}
Notice that according to the values of $c_0$, $c_1$ and $\acc$, the value $w$ can be real (w.l.o.g. assumed positive) or complex and $\theta_1=\theta-\theta_0$.
The constants in the first two lines of \eqref{eq:constants} are used to simplify expressions when $w$ is real, the last line is used in the cases of $w$ complex.
\noindent We are now able to state the main result of this section.
\begin{lemma}\label{lem:vel}
The analytic solution of \eqref{eq:riccati} with initial velocity $\vzero$
takes the following forms:
\begin{EQ}\label{eq:v}
   v(t) = 
   \vzero + \dfrac{(\acczero-\acc)\Efun(t,w)}{1-\gamma\Efun(t,w)}\; w\geq 0;
   \quad
   v(t) =
   \dfrac{\sin\left(\theta_0-\frac{1}{2}t\abs{w}\right)}
         {\sin\left(\theta_1+\frac{1}{2}t\abs{w}\right)}
         \dfrac{\sqrt{\abs{\acc}}}{\sqrt{c_1}},\; w\in\mathbbm{C}\setminus\mathbbm{R}\quad
\end{EQ} 
where $\Efun(t,w)\DEF(1-\mathrm{e}^{wt})/w$ and $\theta$, $\theta_0$ and $\theta_1$ are defined in~\eqref{eq:constants}.
The solution $v(t)$ is meaniningful only for finite nonnegative values.
\end{lemma}
\begin{proof}
It is straightforward to check that \eqref{eq:v} satisfies
the Riccati ODE~\eqref{eq:riccati}.\qed
\end{proof}

As remarked before, the sign of the argument of the square root in $w$ is fundamental for determining the shape of the solution of the ODE. With this notation, the asymptotic velocity, when there is one, can be written for $\acc= \ap$ as $v_{\infty}=\ap/\alpha\geq 0$.

\noindent The following lemma states the interval where the solution is defined.

\begin{lemma}\label{lem:velocity}
  The solution $v(t)$ of the problem~\eqref{eq:riccati}
  is finite and nonnegative for $t$ in the interval summarised in Table~\ref{tab:interval}
  and plotted in Figure \ref{fig:cases:v}.
  The values of $\tzero$, $\tinf$, $\tvzero$ and $\tvinf$ are the following:
  \begin{EQ}[rcll]\label{eq:t0:def}
     \tzero &=& \dfrac{1}{w}\log\Big(1-w\dfrac{\vzero}{\acc+\vzero\beta}\Big),
     \qquad & \vzero < v_{\infty} \textrm{ or } \acc < 0\\
     \tinf &=& \dfrac{1}{w}\log\Big(1-\dfrac{w}{\gamma}\Big)
     \qquad & \vzero > v_{\infty} \textrm{ or } \acc < 0
     \\
     \tvzero &=&
     \dfrac{2\theta_0}{\abs{w}},\qquad
     \tvinf = -\dfrac{2\theta_1}{\abs{w}},
     \qquad & w \textrm{ complex}
  \end{EQ}
  the value of $\theta_0$ and $\theta_1$ are given in~\eqref{eq:constants}.
\end{lemma}
\begin{proof}
  We have 5 cases (summarized in Table~\ref{tab:interval}) that are function of the parameters
  $\acc$, $\vzero$, $c_0$ and $c_1$:
  \begin{description}
    \item[(case a)] This case corresponds to have $ \alpha \vzero = \acc $, the constant solution, i.e. when the r.h.s. 
                    of ODE~\eqref{eq:riccati} vanishes, i.e.
                    $\acc=c_0\vzero+c_1\vzero^2$. Its nonnegative solution is $\vzero = \acc/\alpha$;
    \item[(case b)] This case happens for $\alpha \vzero < \acc$ when the numerator of~\eqref{eq:v} is zero 
                    for some $t$. A simple computation shows that
                    the numerator has a single \emph{real} root for $\alpha \vzero < \acc$ or $\acc <0$.
                    For $\alpha \vzero < \acc$ the value of the root is $\tzero$ of equation~\eqref{eq:t0:def}.
    \item[(case c)] This case happens for $\alpha \vzero>\acc \geq 0 $ when the denominator of~\eqref{eq:v}  
                    is zero for some $t$. A simple computation shows that
                    the denominator has a single \emph{real} root only if $\acc < \alpha \vzero$ and the value of the root  
                    is $\tinf$ of equation~\eqref{eq:t0:def}.
    \item[(case d)] This case happens when the denominator of~\eqref{eq:v} is zero for some $\tinf$
                    and the numerator of~\eqref{eq:v} is zero for some $\tzero$.
                    This can happen only when $\acc<0$.
    \item[(case e)] Equation~\eqref{eq:v} has denominator zero for some $\tvinf$
                    and numerator zero for some $\tvzero$. This zeros are the border of the interval.
  \end{description}
  The remaining part of the proof is a simple computation of the zeros of numerators and denominators.\qed
\end{proof}
The velocity 
has to be positive, but this is true on some intervals only. Table \ref{tab:interval} summarises the intervals of nonnegative velocity in terms of time and also of travelled space.
\begin{table}[!tb]
  \begin{center}
  \caption{Interval of existence and where the solution of~\eqref{eq:riccati} is
           finite and nonnegative}\label{tab:interval}
  \def\arraystretch{1.2}
  \begin{tabular}{|cl|ll|c|}
    \hline
    case   & condition  & $(\tmin,\tmax)$   & $(\smin,\smax)$ & $v'(t)$ \\
    \hline
    (a) & $w\geq 0$ and $\vzero = v_{\infty}$ & $(-\infty,\infty)$ & $(-\infty,\infty)$  & $=0$ \\
    (b) & $w\geq 0$ and $\vzero < v_{\infty}$ & $[\tzero,\infty)$  & $[\szero,\infty)$   & $>0$ \\
    (c) & $w\geq 0$ and $\vzero>v_{\infty}$   & $(\tinf,\infty)$   & $(-\infty,\infty)$  & $<0$ \\
    (d) & $w\geq 0$ and $\acc < 0$            & $(\tinf,\tzero]$   & $(-\infty,\szero]$  & $<0$ \\
    (e) & $w$ complex                         & $(\tvinf,\tvzero]$ & $(-\infty,\svzero]$ & $<0$ \\
    \hline
  \end{tabular}
  \end{center}
\end{table}
We focus now on the other state variable, the space $s$. The integration of~\eqref{eq:v} permits
to define the travelled space and the explicit expression is contained in the next Lemma.
\begin{lemma}\label{lem:space}
Let $s(t) = \int_0^t v(\zeta)\,\mathrm{d}\zeta$
be the space travelled with velocity $v(t)$, then
\begin{EQ}\label{eq:s}
   s(t)
   =
   \begin{cases}
   \vinf t + c_1^{-1}\log\left(1-c_1(\vinf-\vzero)\Efun(-t,w)\right), & \textrm{for $w\geq 0$}\\
   c_1^{-1}\bigg(
     \log\bigg(\dfrac{\sin\left(\theta_1+\frac{1}{2}t\abs{w}\right)}
                     {\sin\theta_1}\bigg)
     -\dfrac{c_0}{2}t
   \bigg), & \textrm{for $w$ complex}
   \end{cases}
\end{EQ}
the value of $\theta_0$ and $\theta_1$ are defined in equation~\eqref{eq:constants}.
\end{lemma}
\begin{proof}
As with Lemma \ref{lem:velocity}, the expressions for $s(t)$ are obtained by standard integration techniques, it is an exercise to check that they satisfy $s'(t)=v(t)$. 
Integrating~\eqref{eq:v} and noticing that $\alpha\beta=\acc c_1$:
\begin{EQ}[rcl]
   s(t)
   &=& \dfrac{\alpha \vzero-\acc}{\beta-c_1\vzero} t+
   \dfrac{\acczero-\acc}{\gamma(\gamma-w)}\log\left(
   1-\gamma\Efun(t,w)
   \right)
   \\
   &=& \dfrac{\alpha \vzero-\acc}{\beta-c_1\vzero}\dfrac{\alpha+c_1\vzero}{\alpha+c_1\vzero} t+
   c_1^{-1}\log\left(
   1-(c_1\vzero+\alpha)\Efun(t,w)
   \right)
   \\
   &=& -c_1^{-1}\alpha t+
   c_1^{-1}\log\left(
   1-(c_1\vzero+\alpha)\Efun(t,w)
   \right).
\end{EQ}
Then using the equality $\vinf c_1=\beta$, we have
\begin{EQ}[rcl]
   s(t)
   &=& \vinf t
   -c_1^{-1}(\alpha+\beta) t+
   c_1^{-1}\log\left(
   1-(c_1\vzero+\alpha)\Efun(t,w)
   \right)
   \\
   &=& \vinf t + c_1^{-1}\log\left(
   \mathrm{e}^{-tw}-(c_1\vzero+\alpha)\mathrm{e}^{-tw}\Efun(t,w)
   \right)
   \\
   &=& \vinf t + c_1^{-1}\log\left(
   1-w\Efun(-t,w)+(c_1\vzero+\alpha)\Efun(-t,w)
   \right)
   \\
   &=& \vinf t + c_1^{-1}\log\left(
   1+(c_1\vzero-\beta)\Efun(-t,w)
   \right)
   \\
   &=& \vinf t + c_1^{-1}\log\left(
   1+(c_1\vzero-c_1\vinf)\Efun(-t,w)
   \right).
\end{EQ}
Now from the next equality, for constant $a$, $b$, $c$ and $d$,
\begin{EQ}
  \dfrac{\mathrm{d}}{\mathrm{d} t}\dfrac{(c\,t-a)\cos(a+b)-\log(\sin(c\,t+b))\sin(a+b)}{c}
  =
  \dfrac{\sin(c\,t-a)}{\sin(c\,t+b)}
\end{EQ}
and by posing $a=\theta_0$, $b=\theta_1$, $c=\abs{w}/2$ with the angle definitions of~\eqref{eq:constants}, we obtain~\eqref{eq:s}.\qed
\end{proof}
%
As a corollary, it is possible to determine the special values $\szero$ and $\svzero$. 
\begin{corollary}\label{cor:3.1}
The values of $\szero$ and $\svzero$
of Table~\ref{tab:interval} are:
\begin{EQ}\label{eq:s0sv0}
  \szero =
  \dfrac{1}{w}\left[\dfrac{\beta}{c_1}\log\left(1-\vzero\dfrac{c_1}{\beta}\right)+
  \dfrac{\alpha}{c_1}\log\left(1+\vzero\dfrac{c_1}{\alpha}\right)\right],
  \qquad
  \svzero =
  \dfrac{1}{2c_1}\left[
   \log\left(1+\dfrac{\acczero}{\abs{\acc}}
   \right)-c_0\tvzero\right],\quad
\end{EQ}
where $\tvzero$ is defined in \eqref{eq:t0:def}.
\end{corollary}
\begin{proof}
We use the relations
\begin{EQ}
   \Efun(-\tzero,w)=
   \dfrac{1}{w}\left(
   1-\dfrac{\gamma \vinf}{\gamma \vinf-w \vzero}\right)
   =\dfrac{\vzero}{w \vzero-\gamma \vinf}
   =\dfrac{\vzero}{w \vzero-\vzero\beta-\acc}
   =\dfrac{\vzero}{\alpha\vzero-\alpha\vinf}
\end{EQ}
and~\eqref{eq:t0:def} in~\eqref{eq:s}, the two values $s(\tzero)$ and $s(\tvzero)$ become respectively:
\begin{EQ}[rcl]
   s(\tzero)
   &=& \dfrac{\vinf}{w}\log\left(1-\dfrac{w}{\gamma}\dfrac{\vzero}{\vinf}\right) 
   + \dfrac{1}{c_1}\log\left(
   1-c_1(\vinf-\vzero)\Efun(-\tzero,w)
   \right)\\
   &=& 
   \dfrac{\vinf}{w}\log\left(1-\dfrac{w}{\gamma}\dfrac{\vzero}{\vinf}\right) 
   + \dfrac{1}{c_1}\log\left(
   1+c_1\dfrac{\vzero}{\alpha}
   \right)
   \\
   &=& 
   \dfrac{\beta}{w c_1}\log\left(1-c_1\dfrac{w}{\gamma}\dfrac{\vzero}{\beta}\right) 
   + \dfrac{\alpha+\beta}{w c_1}\log\left(
   1+c_1\dfrac{\vzero}{\alpha}
   \right)
   \\
   &=& 
   \dfrac{\beta}{w c_1}\log\left(1-c_1\dfrac{w}{\gamma}\dfrac{\vzero}{\beta}\right) 
   + \dfrac{\beta}{w c_1}\log\left(1+c_1\dfrac{\vzero}{\alpha}\right)
   + \dfrac{\alpha}{w c_1}\log\left(1+c_1\dfrac{\vzero}{\alpha}\right)
   \\
   &=& 
   \dfrac{1}{w}\left[
   \dfrac{\beta}{c_1}\log\left(1-\dfrac{c_1}{\beta}\vzero\right) 
   + \dfrac{\alpha}{c_1}\log\left(1+\dfrac{c_1}{\alpha}\vzero\right)\right],
  \end{EQ}
  and
  \begin{EQ}[rcl]
     s(\tvzero)+\dfrac{c_0}{2c_1}\tvzero
     &=&
     \dfrac{1}{c_1}
     \log\left(
     \dfrac{\sin(\theta-\theta_0+\frac{1}{2}\tvzero\abs{w})}{\sin(\theta-\theta_0)}
     \right)
     =
     \dfrac{1}{c_1}
     \log\left(
     \dfrac{\sin\theta}{\sin(\theta-\theta_0)}
     \right) \\
     &=&-
     \dfrac{1}{c_1}
     \log\left(
     \dfrac{\sin\theta\cos\theta_0-\cos\theta\sin\theta_0}{\sin\theta}
     \right)
     =
     -
     \dfrac{1}{c_1}
     \log\left(\cos\theta_0-\dfrac{\sin\theta_0}{\tan\theta}
     \right)
     \\
     &=&
     -
     \dfrac{1}{c_1}
     \log\left(\cos\theta_0-\dfrac{c_0}{\abs{w}}\sin\theta_0
     \right)
     =
     -
     \dfrac{1}{2c_1}
     \log\left(\dfrac{\abs{\acc}}{\acczero+\abs{\acc}}\right).
  \end{EQ}
  Now bringing $\dfrac{c_0}{2c_1}\tvzero$ to the r.h.s. completes the proof.
  \qed
\end{proof}
\begin{lemma}\label{lem:acceleration}
  The acceleration $v'(t)$ of Equation~\eqref{eq:riccati}
  is identically $0$ or is of constant sign in the interval where $v(t)\geq 0$,
  the corresponding signs are collected in Table~\ref{tab:interval}.
  Moreover, the space  $s(t)$ is monotone increasing  where $v(t)\geq 0$
  and thus the inverse function $t(\hat{s})$, i.e. the solution of the problem
  $s(t,\vzero) = \hat{s}$, is well defined.
\end{lemma}
\begin{proof}
  From equation~\eqref{eq:riccati} with constant acceleration $\acc$, suppose that for a certain $t=\tau$  the velocity is $v'(\tau)=0$. If we consider the shifted function 
  $\tilde{v}'(t)=v'(t-\tau)$, we have
  $0=\acc-(c_0-c_1\tilde{v}(0))\tilde{v}(0)$
  and from equation~\eqref{eq:v} of Lemma~\ref{lem:vel}
  the solution~$\tilde{v}(t)$ is constant.\qed
\end{proof}
\begin{corollary}\label{cor:3.6}
  The function $\tilde{v}(s)=v(t(s))$
  satisfies the following ODE:
  \begin{EQ}
    \tilde{v}'(s)=\dfrac{\acc(s)}{\tilde{v}(s)}-c_0-c_1\tilde{v}(s), \quad \tilde{v}(0)=\vzero\geq 0,
  \end{EQ}
  that is, the ODE~\eqref{eq:riccati} in the new independent variable $s$. This change of variable will be used in next sections.
\end{corollary}

\section{Stable Computation of the Analytic Solutions}\label{sec:stable}
Expressions~\eqref{eq:v} and \eqref{eq:s}
may be unstable or numerically inaccurate for $c_0\approx 0$, $c_1\approx 0$ or $w\approx 0$
or for $t\approx \tinf$ or $t\approx \tzero$ .
Numerically accurate reformulation for expressions~\eqref{eq:v} and \eqref{eq:s} are derived in this section.
We define some auxiliary functions:
\begin{EQ}[rclrcl]\label{eq:aux_functions}
  \Lfun(t,c) &\DEF& c^{-1}\log(1-ct),\quad&
  \Efun(t,w) &\DEF& w^{-1}\big(1-\mathrm{e}^{wt}\big),\\
  \Sfun(x,w) &\DEF& w^{-1}\sin(wx),\quad&
  \Afun(x,w) &\DEF& w^{-1}\arctan(wx),
\end{EQ}
and
\begin{EQ}[rcl]\label{eq:Ffun}
  \Ffun(t,w,c_0)
  &\DEF& \mathrm{e}^{tc_0/2}\Gfun(t,w,c_0),\\
  \Gfun(t,w,c_0)
  &\DEF&
  \dfrac{1}{w}\Big[
  \Efun\Big(-t,\dfrac{w+c_0}{2}\Big)+\Efun\Big(t,\dfrac{w-c_0}{2}\Big)
  \Big],
\end{EQ}
\begin{lemma}\label{lem:4.1}
The functions \eqref{eq:aux_functions} and \eqref{eq:Ffun} are smooth in their definition domain and
have the following converging Taylor expansions:
\begin{EQ}[rclrcl]\label{eq:Ffun:serie1}
   \Lfun(t,c) &= & -t\sum_{n=0}^\infty \dfrac{(ct)^n}{n+1},
   \; &
   \Efun(t,w) &= & -t\sum_{n=0}^\infty \dfrac{(wt)^{n}}{(n+1)!},
   \\
   \Sfun(x,w) &=& x\sum_{n=0}^\infty\dfrac{(-(wx^2))^n}{(2n+1)!},
   \; &
   \Afun(x,w) &=& x\sum_{n=0}^\infty\dfrac{(-(wx)^2)^n}{2n+1},
   \\
   \Ffun(t,w,c_0) &=&
   -\sum_{n=1}^\infty\dfrac{f_n t^{2n}}{2^{2n}(2n)!}
   \left[4+\dfrac{2c_0t}{2n+1}\right],\; &
   f_n &=& \dfrac{w^{2n}-c_0^{2n}}{w^2-c_0^2},
   \\
   \Gfun(t,w,c_0) &=& -\sum_{n=1}^\infty f_n\mathrm{e}^{h_n(t)} \left[ 4+\dfrac{2c_0t}{2n+1}\right],\; &
   h_n(t) &=& \sum_{j=1}^{2n}\log\dfrac{\abs{t}}{2j}-\dfrac{tc_0}{2},
\end{EQ}
and the term $f_n$ is computed by the following recurrence:
\begin{EQ}
  \textrm{Initial value:}\quad
  \begin{cases} g_1 = 1, \\ f_1 = 1,\end{cases}
  \qquad
  \textrm{Recurrence:}\quad
  \begin{cases} g_{n+1} = g_{n}c_0^{2},\\f_{n+1} = f_{n}w^2 + g_{n+1}.\end{cases}
\end{EQ} 

\end{lemma}
\begin{proof}
The Taylor series for $\Lfun$, $\Efun$, $\sinc$\footnote{Notice that we use the \emph{unnormalised sinc function}, defined, for $x\neq 0$ and extended for $x=0$ with its limit.}, $\Afun$ are straightforward. %
Some effort is necessary to derive the expansion for $\Ffun(t,w,c_0)$ and $ \Gfun(t,w,c_0)$.
To compute~\eqref{eq:Ffun}, we multiply it by $(w^2-c_0^2)/2$
so that we obtain:
\begin{EQ}[rcl]
   \dfrac{w^2-c_0^2}{2}
   \Ffun(t,w,c_0)
   &=&
   2\mathrm{e}^{tc_0/2}-\mathrm{e}^{-tw/2}-\mathrm{e}^{tw/2}+
   \dfrac{c_0}{w}\left(\mathrm{e}^{-tw/2}-\mathrm{e}^{tw/2}\right)
   \\
   &=&
   \sum_{n=1}^\infty
   \dfrac{t^n}{2^n n!}
   \left(
   2c_0^n-((-w)^n+w^n)+\dfrac{c_0}{w}\left((-w)^n-w^n\right)
   \right)
   \\
   &=&
   \sum_{n=3,5,7,\ldots}^\infty
   \dfrac{t^n\left(2c_0^n-2c_0w^{n-1}\right)}{2^n n!}
   +
   \sum_{n=2,4,5,\ldots}^\infty
   \dfrac{t^n\left(2c_0^n-2w^n\right)}{2^n n!}
   \\
    &=&
   \sum_{n=1}^\infty
   \dfrac{t^{2n+1}c_0\left(c_0^{2n}-w^{2n}\right)}{2^{2n}(2n+1)!}
   +
   \sum_{n=1}^\infty
   \dfrac{t^{2n}\left(c_0^{2n}-w^{2n}\right)}{2^{2n-1}(2n)!}
   \\
    &=&
   \sum_{n=1}^\infty
   \dfrac{t^{2n}}{2^{2n}(2n)!}
   \left[
   2+\dfrac{c_0t}{2n+1}
   \right]
   \left(
   c_0^{2n}-w^{2n}
   \right).
\end{EQ}
If we define the general term $f_n \DEF (w^{2n}-c_0^{2n})/(w^2-c_0^2)$
then \eqref{eq:Ffun:serie1} is readily obtained.
Those functions are smooth and numerically stable even when the arguments $w$ and $c_0$ 
approach zero. Notice that an efficient computation of the series~\eqref{eq:Ffun}
uses the recurrence for $f_n$.
The use of the recurrence is mandatory to avoid cancellation errors or numerical overflows, hence the stable and efficient algorithm to evaluate~\eqref{eq:Ffun:serie1} is summarised in the next equation:
\begin{EQ}
   \Gfun(t,w,c_0)
   =
   -\sum_{n=1}^\infty
   \mathrm{e}^{h_n(t)}
   \left[
   4+\dfrac{2c_0t}{2n+1}
   \right],
   \qquad
   h_n(t) = 2n\log\dfrac{\abs{t}}{2}-\dfrac{tc_0}{2}-\log((2n)!),
\end{EQ}
where
\begin{EQ}
  h_n(t) = 2n\log\dfrac{\abs{t}}{2}-\dfrac{tc_0}{2}-\sum_{j=1}^{2n}\log j
      = \sum_{j=1}^{2n}\left(\log\dfrac{\abs{t}}{2}-\log j\right)-\dfrac{tc_0}{2}
      = \sum_{j=1}^{2n}\log\dfrac{\abs{t}}{2j}-\dfrac{tc_0}{2}.
\end{EQ}
and the remaining part follows easily.\qed
\end{proof}

The computation of the special values  $\tzero$, $\tinf$,
$\tvzero$, $\vinf$ of Table~\ref{tab:interval} and corresponding values of $s(t,v)$ is critical and must be accurate, otherwise complex arguments or logarithms of negative values are very likely. Using~Lemma~\ref{lem:4.1}, the stable computation of the ranges in \eqref{eq:t0:def}
is written in terms of the smooth functions introduced in
\eqref{eq:Ffun:serie1}: 
\begin{EQ}[rclrcll]\label{eq:t0tv0:def}
  \tzero &=& \Lfun\Big(\dfrac{\vzero}{\acc+\vzero\beta},w\Big),
  \quad&
  \tinf &=& \Lfun\Big(\dfrac{1}{\gamma},w\Big),\quad &
  w\geq 0
  \\
  \tvzero &=& 2\Afun\Big(\dfrac{\vzero}{\vzero c_0+2\abs{\acc}},\abs{w}\Big),
  \quad&
  \tvinf&=& \tvzero-2\Afun\Big(\dfrac{1}{c_0},\abs{w}\Big),\quad &
  \textrm{$w$ complex.}
\end{EQ}
An algorithmic version of the computation of the necessary constants for the semi-analytic
solution of the OCP \eqref{OCP} is given in Algorithm \ref{algo:setup}.
\subsection{Stable speed computation}
With the previously introduced formulas, the expressions of Lemma \ref{lem:vel} 
are rewritten in the numerically stable version in the next Proposition.

\begin{proposition}
  The speed $v(t)$ is numerically stable for parameters 
  $w\approx 0$ or $t\approx \tzero$ or $t\approx\tinf$ if computed as follows (An
  algorithmic version for computing $v(t)$ is presented in Algorithm \ref{algo:velocity}).
  For $w$ real, the velocity is rewritten as:
  \begin{EQ}\label{eq:lem6:1}
     v(t) = \dfrac{p(t)}{q(t)}+
     \begin{cases}
       0      & \textrm{if $\abs{t-\tzero}\leq \epsilon$}, \\
       \vzero & \textrm{otherwise},
     \end{cases}
     \;
     \left\{\begin{array}{r@{~}c@{~}l}
     p(t) &=& \begin{cases} 
         (\alpha \vzero-\acc)\Efun(t-\tzero,w) & \textrm{if $\abs{t-\tzero}\leq \epsilon$}, \\
         (\acczero-\acc)\Efun(t,w)             & \textrm{$t>0$}, \\
         (\acc-\acczero)\Efun(-t,w)            & \textrm{$t\leq 0$},
         \end{cases}
     \\[2em]
     q(t) &=& \begin{cases}
         (w-\gamma)\Efun(t-\tinf,w) & \textrm{if $\abs{t-\tinf}\leq\epsilon$}, \\
         1-\gamma\Efun(t,w)         & \textrm{$t>0$}, \\
         1+(\gamma-w)\Efun(-t,w)    & \textrm{$t\leq 0$}.
         \end{cases}
     \end{array}\right.
  \end{EQ}
  When $w$ is complex, a stable computation is:
  \begin{EQ}\label{eq:lem6:2}
     v(t)=
     \dfrac{\vzero\cos(\frac{1}{2}t\abs{w})-(c_0\vzero+2\abs{\acc})\Sfun(\frac{1}{2}t,\abs{w})}
     {\cos(\frac{1}{2}t\abs{w})+(2c_1\vzero+c_0)\Sfun(\frac{1}{2}t,\abs{w})}
  \end{EQ}
\end{proposition}
\begin{proof}
\emph{Case $w\geq 0$ real}. Notice that from~\eqref{eq:v} it is possible to write for $v$,
\begin{EQ}
  \vzero + \dfrac{(\acczero-\acc)\Efun(t,w)}{1-\gamma\Efun(t,w)}
  = \dfrac{\vzero-(\acc+\vzero\beta)\Efun(t,w)}{1-\gamma\Efun(t,w)}
  = \vzero + \dfrac{(\acc-\acczero)\Efun(-t,w)}{1+(\gamma-w)\Efun(-t,w)},
\end{EQ}
therefore for the numerator, we have:
\begin{EQ}[rcl]
   \vzero-(\acc+\vzero\beta)\Efun(t,w) 
   &=&
   \vzero-(\acc+\vzero\beta)
   w^{-1}\big(\mathrm{e}^{-\tzero w}-\mathrm{e}^{(t-\tzero)w}\big)
   \mathrm{e}^{\tzero w} \\
   &=&
   \vzero-(\acc+\vzero\beta)
   \left(-\Efun(-\tzero,w)+\Efun(t-\tzero,w)\right)\mathrm{e}^{\tzero w}\\
   &=&
   \vzero-(\acc+\vzero\beta)
   \Big(\Efun(\tzero,w)+\Efun(t-\tzero,w)\Big(1-\dfrac{w}{\gamma}\dfrac{\vzero}{\vinf}\Big)\Big)
   \\
   &=&
   \vzero-(\acc+\vzero\beta)
   \Big(\dfrac{1}{\gamma}\dfrac{\vzero}{\vinf}+\Efun(t-\tzero,w)\Big(1-\dfrac{w}{\gamma}\dfrac{\vzero}{\vinf}\Big)\Big)
   \\
   &=&
   \Efun(t-\tzero,w)\left(\acc+\vzero\beta-w\vzero\right)\qquad [\acc+\vzero\beta=\gamma\vinf]
   \\
   &=&
   \Efun(t-\tzero,w)\left(\acc-\alpha\vzero\right)
\end{EQ} 
And for the denominator,
\begin{EQ}[rcl]
   1-\gamma\Efun(t,w) 
   &=&
   1-\gamma w^{-1}\big(\mathrm{e}^{-\tinf w}-\mathrm{e}^{(t-\tinf)w}\big)\mathrm{e}^{\tinf w}
   \\
   &=&
   1-\gamma
   \left(-\Efun(-\tinf,w)+\Efun(t-\tinf,w)\right)\mathrm{e}^{\tinf w}\\
   &=&
   1-\gamma
   \big(\Efun(\tinf,w)+\Efun(t-\tinf,w)\big(1-w\gamma^{-1}\big)\big)
   \\
   &=&
   1-\gamma
   \big(\gamma^{-1}+\Efun(t-\tinf,w)\big(1-w\gamma^{-1}\big)\big)
   \\
   &=&
   \Efun(t-\tinf,w)\left(w-\gamma\right).
\end{EQ}

\emph{Case $w$ complex}. The standard expansion of the trigonometric functions in \eqref{eq:v} yields immediately
\eqref{eq:lem6:2}.\qed
\end{proof}

\subsection{Stable space computation}
We restate here the numerically stable formulas of Lemma \ref{lem:space}.
\begin{lemma}
  The space $s(t)$ of formula~\eqref{eq:s} is numerically stable for parameters 
  $w\approx 0$, or $t\approx \tzero$ or $t\approx\tinf$, if computed as follows.
  For $w$  real, $s(t)$ can be rewritten as:
  \begin{EQ}\label{eq:lem6:1:b}
    s(t) = \vinf t+
    \Lfun\left((\vinf-\vzero)\Efun(-t,w),c_1\right)
    =
    \Lfun\left(
    \acc\Gfun(t,w,c_0)
    -\vzero\Efun(-t,w)\mathrm{e}^{\beta t},c_1
    \right),
  \end{EQ}
  where we use the second equation when $\vinf\gg\vzero$.
  For $w$ complex and $c_1\gg0$ a stable computation is:
  \begin{EQ}\label{eq:lem6:2:b}
     s(t) =
     \dfrac{1}{c_1}
     \log\Big(
      (2c_1\vzero+c_0)\Sfun\Big(\frac{t}{2},\abs{w}\Big)+\cos\frac{t\abs{w}}{2}
     \Big)
     -\dfrac{c_0 t}{2c_1},
  \end{EQ}
  the case $c_1\approx 0$ is considered in the next lemma.
\end{lemma}
\begin{proof}
\emph{Case $w\geq 0$ real}. In this case equation \eqref{eq:v} becomes, after some manipulations:
\begin{EQ}[rcl]
   s(t)
   &=&
   \vinf t + c_1^{-1}\log\left(1-c_1(\vinf-\vzero)\Efun(-t,w)\right)
   \\
   &=&
   c_1^{-1}\beta t + c_1^{-1}\log\left(1-c_1(\vinf-\vzero)\Efun(-t,w)\right), \\
   &=&
   c_1^{-1}
   \log\left( \mathrm{e}^{\beta t}-c_1(\vinf-\vzero)\Efun(-t,w)\mathrm{e}^{\beta t}\right)
   \\
   &=&
   c_1^{-1}
   \log\left( 1-\beta\Efun(t,\beta)-c_1(\vinf-\vzero)\Efun(-t,w)\mathrm{e}^{\beta t}\right),\\
   &=&
   c_1^{-1}
   \log\left(
   1-c_1\vinf\left(
   \Efun(t,\beta)+
   w^{-1}\big(\mathrm{e}^{\beta t}-\mathrm{e}^{-\alpha t}\big)\right)
   +c_1\vzero
   \Efun(-t,w)\mathrm{e}^{\beta t}
   \right),\\
   &=&
   c_1^{-1}
   \log\left(
   1-c_1\acc
   w^{-1}\big(\Efun(t,\beta)+\Efun(-t,\alpha)\big)
   +c_1\vzero
   \Efun(-t,w)\mathrm{e}^{\beta t}
   \right).\\
\end{EQ}
\emph{Case $w$ complex}. Follows easily using~\eqref{eq:s} and the angle definitions in~\eqref{eq:v}.\qed
\end{proof}

%
\begin{lemma}
  Computation of $s(t)$ when $w$ is complex and $c_1\approx 0$:
  \begin{EQ}\label{eq:s:stable}
     s(t) =\begin{cases}
     (\acczero+\abs{\acc})t^2 Q(t\ell_1,\cos\theta_1)+\vzero t
     & \textrm{$\abs{t\ell_1} \leq 0.001$} \\
     \textrm{use \eqref{eq:lem6:2:b}}
     & \textrm{$\abs{t\ell_1}>0.001$}
     \end{cases}
  \end{EQ}
  where
  \begin{EQ}\label{eq:def:ell1}
     \ell_1       \DEF \sqrt{c_1(\acczero+\abs{\acc})}, \qquad
     \cos\theta_1 \DEF \dfrac{2c_1\vzero+c_0}{2\ell_1}, \qquad
     \sin\theta_1 \DEF \dfrac{\abs{w}}{2\ell_1},
  \end{EQ}
  and $Q(\tau,c)$ can be approximated (when $t\ell_1\approx 0$) with 
  \begin{EQ}
   Q(\tau,c)=-\dfrac{1}{2}
    +\dfrac{c\tau}{3}
    -\dfrac{2c^2+1}{12}\tau^2
    +\dfrac{c^2+2}{15}c\tau^3
    -\dfrac{2c^4+11c^2+2}{90}\tau^4
    +\dfrac{2c^4+26c^2+17}{315}c\tau^5
    -\varepsilon(\tau,c).
  \end{EQ}
  The remainder is $\abs{\varepsilon(\tau,c)}\leq 10^{-18}$.
\end{lemma}
\begin{proof}
Using~\eqref{eq:s}, the definitions of the constants~\eqref{eq:def:ell1} and the angles $c_0/2=\ell_1\cos\theta_1-c_1\vzero$ we have
\begin{EQ}
   s(t)
   =
   \dfrac{\acczero+\abs{\acc}}{\ell_1^2}
   \left(
   \log\left(
   \dfrac{\sin\left(t\ell_1\sin\theta_1\right)}{\tan\theta_1}+\cos\left(t\ell_1\sin\theta_1\right)
   \right)
   -t\ell_1\cos\theta
   \right)+\vzero t.
\end{EQ}
By posing $\tau=t\ell_1$ and expanding with Taylor around $\tau$:
\begin{EQ}
   \Qfun(\tau,c)=
   \dfrac{1}{\tau^2}
   \log\left(
   \dfrac{c}{s}\sin\left(\tau s\right)+\cos\left(\tau s\right)
   \right)-\dfrac{c}{\tau}
   ,\quad s=\sqrt{1-c^2}.
\end{EQ}
The remainder can be estimated with
\begin{EQ}[rcl]
   \epsilon(\tau,c) &=&
   \dfrac{4c^6+114c^4+180c^2+17}{2520}\tau_\star^6, \qquad
   -\dfrac{\pi}{2} \leq \tau \leq \tau_\star \leq 0.
\end{EQ}
From the interval of definition $\abs{\tau}\leq 0.001$, 
we have that $\abs{\epsilon(\tau,c)}\leq 0.125 \tau_\star^6$
for $\abs{\tau}\leq 0.001$, thus the error satisfies $\abs{\epsilon(\tau,c)} \leq 1.25\times 10^{-19}$.\qed
\end{proof}

\begin{remark}
The stable computation of~\eqref{eq:s0sv0} in Corollary~\ref{cor:3.1}
for $w\approx 0$ can be done using \eqref{eq:lem6:1:b} or~\eqref{eq:s:stable}
with~\eqref{eq:t0tv0:def}. An algorithmic version of the stable computation of
the function $s(t)$ is given in Algorithm \ref{algo:space}.
\end{remark}

\section{Velocity as a function of the space $s$}\label{sec:v_of_s}
It is useful in the applications of this work, to have the velocity expressed as a function
of the space and not of the time. This is the case for example when we want to add to the pure Bang-Bang problem a limitation on the lateral acceleration or constrain the problem with the friction ellipse. Those constraints are tipically functions of the space (e.g. the curvature of the trajectory). From Corollary~\ref{cor:3.6}, the function $\tilde{v}(s)$ and its derivatives
can be computed via $t(s)$ with the relation $\tilde{v}(s) = v(t(s))$.
\begin{figure}[!tb]
  \begin{center}
    \begin{tabular}{cc}
      \subfigure{\label{fig:cases:s:b}\includegraphics[scale=0.9]{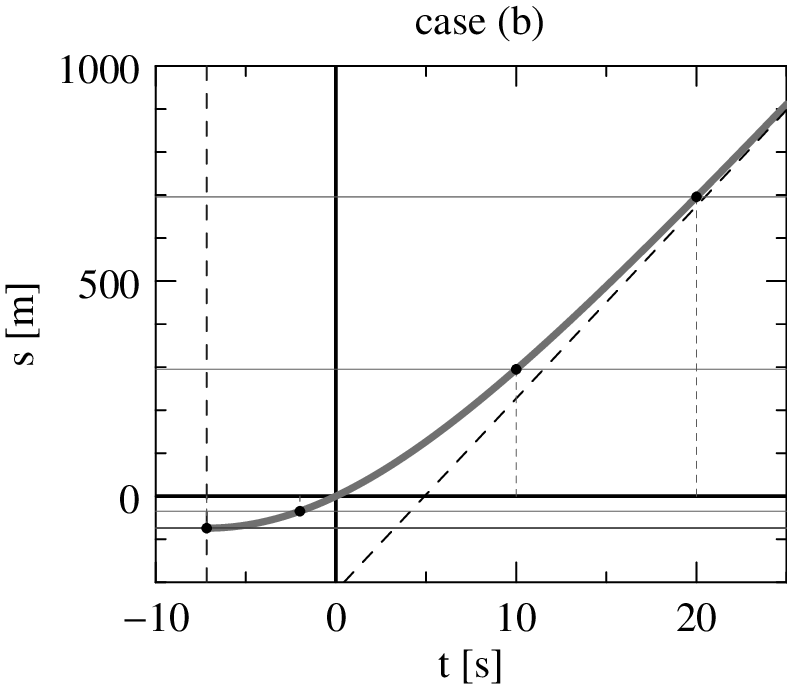}}
      &
      \subfigure{\label{fig:cases:s:c}\includegraphics[scale=0.9]{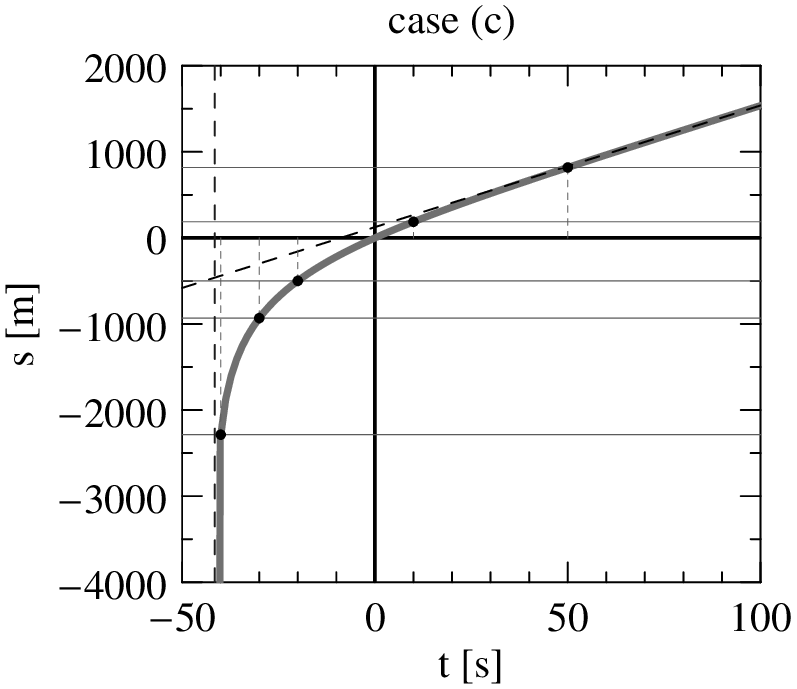}}%
      \\
      \subfigure{\label{fig:cases:s:d}\includegraphics[scale=0.9]{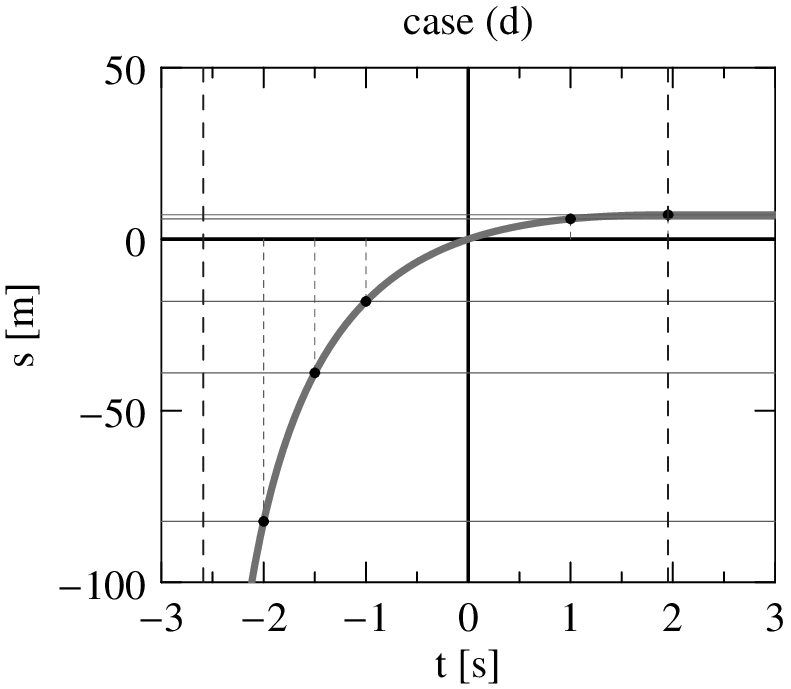}}
      &
      \subfigure{\label{fig:cases:s:e}\includegraphics[scale=0.9]{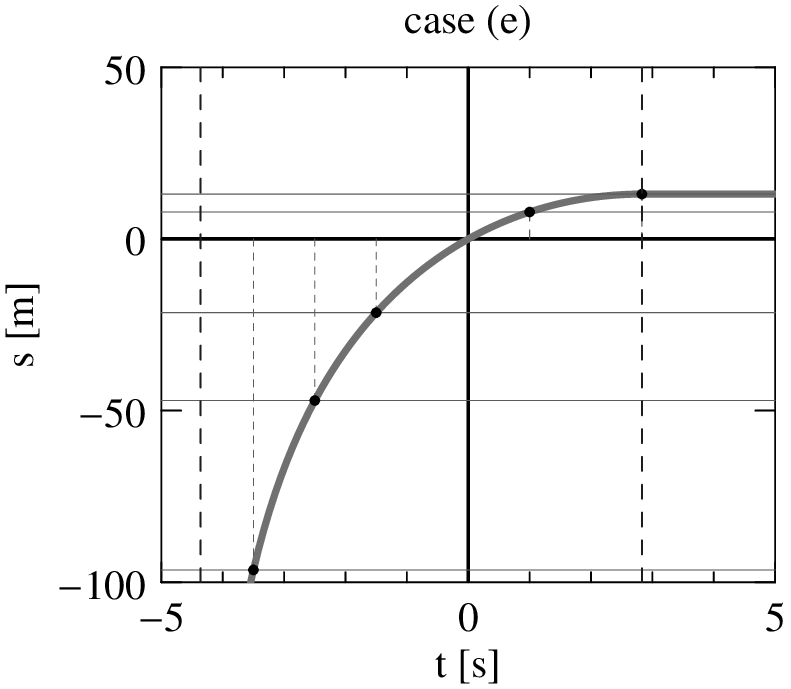}}%
      \end{tabular}
   \end{center}
   \caption{The five cases of the velocity and travelled space as listed in Lemma \ref{lem:velocity}. In all those case the acceleration
                 is constant and positive with constant control $\acc=\ap$.}
   \label{fig:cases:s}
   \end{figure}
The point of the problem is to invert the monotone function $s(t)$, or, in other words, to solve the following equation:
\begin{EQ}\label{eq:stdef}
   t(\zeta)\textrm{ solution of the problem $s(t) =\zeta$ }.
\end{EQ}
The following well-known theorem gives enough conditions for the global convergence of Newton’s method.
\begin{theorem}
  Let $f(x)$ be twice continuously differentiable on the closed finite 
  interval $[a,b]$ and let the following conditions be satisfied:
  \begin{enumerate}
    \item $f(a)f(b)<0$ ;
    \item $f'(x)\neq 0$ for all $x\in[a,b]$
    \item $f''(x)$ either $\geq 0$ or $\leq 0$ for all $x\in[a,b]$
    \item At the endpoint $a$, $b$ both $\abs{f(a)}/\abs{f'(a)}< b-a$ and  $\abs{f(b)}/\abs{f'(b)}< b-a$
          are satisfied;
  \end{enumerate}
  Then Newton's method converges to the unique solution in $[a,b]$ for any choice of the
  initial guess in $[a,b]$.
\end{theorem}
\begin{proof}
  See reference~\cite{Conte:1980}, Theorem 3.2 of page 104.\qed
\end{proof}
\begin{corollary}\label{cor:13}
  Let $f(x)$ be twice continuously differentiable on the closed finite 
  interval $[a,b]$ with
  \begin{enumerate}
    \item $f(a)f(b)<0$ ;
    \item $f'(x)\neq 0$ for all $x\in[a,b]$
    \item $f''(x)\geq 0$ for all $x\in[a,b]$
  \end{enumerate}
  if $x_0\in[a,b]$ satisfies $f(x_0)\geq 0$ then the sequence
  generated by Newton's method converges monotonically to 
  the unique solution in $[a,b]$.
  Analogously if  $f''(x)\leq 0$ for all $x\in[a,b]$ and $f(x_0)\leq 0$.
\end{corollary}
\begin{proof}
  Omitted
\end{proof}

From problem \eqref{eq:stdef} the evaluation of the inverse function $t(\zeta)$
is the solution of the problem $f(t)=0$ where 
$f(t) = s(t) - \zeta$. The function $f(t)$ has the following properties depending on 
the 5 cases (summarized in Table~\ref{tab:interval}) that are function of the parameters
$\acc$, $\vzero$, $c_0$ and $c_1$:
\begin{description}
  \item[case (a)] $f(t)$ is linear and not constant; 
  \item[case (b) up to case (e)] $f'(t)\neq 0$  and  $f''(t)\geq0$ or $f''(t)\leq0$ for all $t\in(\tmin,\tmax)$.
                 Moreover for $\epsilon$ small enough $f(\tmin+\epsilon)f(\tmax-\epsilon) < 0$.
\end{description}

\begin{remark}
The function $f(t) = s(t) - \zeta$ in all cases (a) up to (e) has a unique $t$ that satisfies $f(t)=0$.
Moreover, Newton's method converges monotonically to the solution provided that initial guess $t_{\mathrm{guess}}$ 
is chosen accordingly to Corollary~\ref{cor:13}.
\end{remark}
A special initial guess for Newton iteration applied to the problem $f(t)=0$ in case (b)
is provided using the solution of the Riccati differential
equation with parameters $c_0=c_1=0$ when $\zeta\acc>0$:
\begin{EQ}
  \tguess = \begin{cases}
    \dfrac{2\zeta}{\vzero+\sqrt{4\zeta\acc+\vzero^2}} & \zeta\acc > 0 \\
    0 & \textrm{otherwise.}
  \end{cases}
\end{EQ}
Because $f(t)$ is convex, $f(\tguess)\geq 0$ and
using Corollary~\ref{cor:13} the convergence is monotone.
In the cases $(c)$ $(d)$ and $(e)$ a direct application of the Newton method 
with $\tguess=0$ can cause the iterations to exit the interval of definition of the function, 
$(\tmin,\tmax)$. To avoid this dangerous behaviour,
without using a damped Newton method, we construct the iterations as follows.
From the fact that $\lim_{t\to \tmin} f(t) = -\infty$, we can approximate the function $f(t)$
with a sequence of model functions $g_k(t)$:
\begin{EQ}
   g_k(t) = a_k - \dfrac{b_k}{t-\tmin},\quad
  a_k = f'(t_k)(t_k-\tmin)+f(t_k),\quad
  b_k = f'(t_k)(t_k-\tmin)^2,
\end{EQ}
where $a_k$ and $b_k$ were computed imposing $g_k(t_k)=f(t_k)$ and  $g'_k(t_k)=f'(t_k)$.
The step $t_{k+1}$ is obtained solving $g_k(t_{k+1})=0$:
\begin{EQ}
   t_{k+1} 
   = \tmin+\dfrac{b_k}{a_k}
   = \tmin+\dfrac{f'(t_k)(t_k-\tmin)^2}{f'(t_k)(t_k-\tmin)+f(t_k)}
   = t_k-\dfrac{f(t_k)}{f'(t_k)+\dfrac{f(t_k)}{t_k-\tmin}}.
\end{EQ}
From this construction, if $f(t_k)>0$ and $t_k>\tmin$ then $\tmin < t_{k+1} < t_k$,
if $f(t_k)<0$, then we are in the hypotheses of Corollary~\ref{cor:13} and we can safely employ
Newton's method that gives monotone convergence. In conclusion we model the iterations as stated in Algorithm~\ref{algo:newton}.

\begin{lemma}[c-d-e]
  In the cases (c)--(d)--(e) Algorithm~\ref{algo:newton}
  converges to the solution $f(\troot)=0$ with quadratic convergence.
\end{lemma}
\begin{proof}
  If $f(t_{k_0})<0$ for a certain $k_0$, then the hypotheses of Corollary~\ref{cor:13} are respected,
  hence the quadratic monotone convergence follows from the convergence of Newton's method.   
  Otherwise, let  $\troot$ be the root of $f(\troot)$, because the steps are negative, we have a 
  monotone decreasing sequence bounded from below, which admits the limit $\bar{t}$.
  This limit cannot be $\bar{t}>\troot$ because  $f(\bar{t})>0$, hence the step remains different from zero. Therefore let $\varepsilon_k = t_k-\troot$, then we have
  \begin{EQ}[rcl]
     \varepsilon_{k+1} 
     &=& \varepsilon_k - \dfrac{f(t_k)-f(\troot)}{f'(t_k)+\delta_k^+(f(t_k)-f(\troot))/(t_k-\tmin)}
     \\
     &=& \varepsilon_k - \dfrac{f(\troot)+\varepsilon_kf'(t_k)+\frac{1}{2}\varepsilon_k^2f''(\omega_k)-f(\troot)}
                               {f'(t_k)+\varepsilon_k\delta_k^+ f'(\zeta_k)/(t_k-\tmin)}
     \\
     &=& \frac{1}{2}\varepsilon_k^2\dfrac{f''(\omega_k)(t_k-\tmin)+2\delta_k^+ f'(\zeta_k)}
                                         {f'(t_k)(t_k-\tmin)+\varepsilon_k\delta_k^+ f'(\zeta_k)},
  \end{EQ}
  that is,
  \begin{EQ}
    \lim_{k\to\infty} \dfrac{\varepsilon_{k+1} }{\varepsilon_k^2} = 
    \frac{1}{2}\dfrac{f''(\troot)(\troot-\tmin)+2\delta_k^+ f'(\troot)}
                                         {f'(\troot)(\troot-\tmin)}
    =
    \frac{1}{2}\dfrac{f''(\troot)}{f'(\troot)}+\dfrac{\delta}{\troot-\tmin} = C,
  \end{EQ}
  Hence the convergence is quadratic with constant $C$, which is bigger if the covergence is from the right side.\qed
\end{proof}
\begin{remark}
It is also possible to use higher order Newton methods, for example, for order 3 or 4 
see~\cite{Chun:2009}. The second derivative $f''$ is available analytically, thus the Halley 
method can be applied too~\cite{Neta:2013} (or~\cite{Neta:2014} for a survey of one dimensional methods), for higher order methods see for example~\cite{Amat:2008,Lotfi:2014}.
\end{remark}
\section{The Optimal Control Problem in the curvilinear abscissa $s$}\label{sec:ocp:s}
With the change of coordinates  from time $t$ to space $s$ as independent variable,
it is possible to reformulate the OCP~\eqref{OCP}, passing from free time to a fixed domain.
The change of variable is possible in practice using the results of the 
Section \ref{sec:v_of_s}. Let $\tilde v(s)=v(t(s))$ be the transformed function of
the velocity parametrised as a function of $s$ as showed in Corollary \ref{cor:3.6}, 
then the minimum time optimal control problem \eqref{OCP} can be reformulated as 
finding $a(s)\in[-\ab,\ap]$ that minimises 
\begin{EQ}[rcl]\label{OCP:s}
  \textrm{Minimise}
  \quad T &=& \int_0^L \dfrac{\ds}{\abs{\tilde v(s)}} \\
  \textrm{subject to:}\quad
   \tilde v'(s) &=& \dfrac{\acc(s)}{\tilde v(s)}-c_0-c_1\tilde v(s),\quad
   \tilde v(0) = v_i,\quad \tilde v(L) = v_f,\quad
   -\ab\leq \acc(s) \leq \ap,
\end{EQ}
The complete solution of the Optimal Control Problem~\eqref{OCP:s} is reduced
to finding the optimal switching point $\ssw$ such that the solution is written as
\begin{EQ}
   \tilde v(s) = \begin{cases}
     \tilde v_L(s) & s < \ssw \\
     \tilde v_R(s) & s \geq \ssw,
   \end{cases}
\end{EQ}
where $\tilde v_L(s)$ and $\tilde v_R(s)$ are the solutions of the ODE 
\begin{EQ}[rcrcll]\label{eq:odes_v_s}
   \tilde v_L'(s) &=& \dfrac{\ap}{\tilde v_L(s)}-c_0-c_1\tilde v_L(s),\quad &\tilde v_L(0)&=&v_i, \\
   \tilde v_R'(s) &=& -\dfrac{\ab}{\tilde v_R(s)}-c_0-c_1\tilde v_R(s),\quad& \tilde v_R(L)&=&v_f .\\
\end{EQ}
\begin{remark}
Indeed we notice that while it is possible to solve the dynamic system of \eqref{OCP}
as explicit functions of $v(t)$ and $s(t)$, the differential equation \eqref{OCP:s} 
for $\tilde v(s)$ admits only an implicit solution that is impractical to subsolve 
with respect to $s$. Therefore we employ the knowledge of the analytic solutions 
presented in the previous sections for $v(t)$ and $s(t)$ together with the numeric change of
variable of Section \ref{sec:v_of_s}.
\end{remark}

The computation of the switching point $\ssw$ is done equating the arcs of positive acceleration
with the corresponding arcs of negative acceleration.
It results a single nonlinear equation $g(\ssw)=0$ where:
\begin{EQ}
  g(s) = \tilde v_L(s)-\tilde v_R(s),
\end{EQ}
To solve $g(\ssw)=0$ we simply use the Newton method: the derivative of $g(s)$ is readily given by
\begin{EQ}
  g'(s)=\tilde v'_L(s)-\tilde v'_R(s) =
  \dfrac{\ap}{\tilde v_L(s)}+\dfrac{\ab}{\tilde v_R(s)}+c_1\left(\tilde v_R(s)-\tilde v_L(s)\right)
\end{EQ}
and the latter is deduced from~\eqref{eq:odes_v_s}.
\section{Numerical Tests}\label{sec:tests}
We present herein four numerical tests for various choices of the parameters of the problem,
namely we consider the parameters $c_0$, $c_1$, $\ap$ and $\ab$. 
The tests are done with the stable formulas for velocity and space of Section~\ref{sec:stable} 
with the change of variable from $t$ to $s$ explained in Section~\ref{sec:v_of_s}.\\
It is worth noticing and pointing out that although the solutions of the Riccati differential
equation derived in Section~\ref{sec:solution} are analytic, a direct na\"ive implementation
of such equations leads to numerical instabilities, which is the main motivation that led to
this study. In particular, things go wrong when the friction parameters $c_0$ and $c_1$ are near 
to zero, producing a large error in the division by (near) zero coefficients, or, 
even if they are far from being zero, but $w$ results zero or almost zero, 
for example when the acceleration is negative (braking phase) and 
$c_0^2-4\acc c_1 \approx 0$. The instabilities are also evident in the computation of the 
time and space domains of the equations.
\begin{table}[!tb]
  \begin{center}
  \caption{Table of fixed and varying parameters}\label{tab:2}
  \def\arraystretch{1.2}
  \begin{tabular}{c}
  \begin{tabular}{lllllll}
    \multicolumn{7}{c}{Fixed Parameters} \\ \hline
    $v(0)$ \footnotesize [m/s] &
    $v(L)$ \footnotesize[m/s] &
    $L$\footnotesize [m] & 
    $\ap$ \footnotesize [m/s$^2$] & 
    $\ab$ \footnotesize [m/s$^2$] &
    $c_0$ \footnotesize [1/m] &
    $c_1$ \footnotesize [1/s] \\ \hline
    $6$ & $5$ & $100$ & $2$ & $2$ & $0.01$ & $0.01$ \\
    \hline
  \end{tabular}
  \\[2em]
  \begin{tabular}{c|lllllllll}
    \multicolumn{10}{c}{Varying Parameters} \\ \hline
    $c_0$ \footnotesize [1/m]     & $0$ & $10^{-5}$ & $0.01$ & $0.05$ & $0.1$ & $0.2$ & $0.3$ & $0.4$ & $0.5$ \\ \hline
    $c_1$ \footnotesize [1/s]     & $0$ & $0.005$ & $0.01$ & $0.02$ & $0.03$ \\ \hline
    $\ap$ \footnotesize [m/s$^2$] & $10^{-6}$ & $0.01$ & $0.05$ & $0.1$ & $0.25$ & $1$ & $2$ & $10$ \\ \hline
    $\ab$ \footnotesize [m/s$^2$] & $10^{-6}$ & $0.01$ & $0.05$ & $0.1$ & $0.25$ & $1$ & $2$ & $10$ \\ \hline
  \end{tabular}
  \end{tabular}
  \end{center}
\end{table}

In the first test we  fix the initial/final positions with their speeds and we compute the time
optimal control problem with different values of the friction coefficient $c_0$. The fixed parameters
and the coefficients for the aerodynamic drag are resumed in Table~\ref{tab:2}. The result is 
showed in Figure~\ref{fig:test}.1, the curves are ordered from top to bottom according to the 
increasing value of $c_0$: the top most curve is for $c_0=0$ and the last curve on 
the bottom is the one for $c_0=0.5$. Starting from above, the first five curves represent 
a typical situation of a positive acceleration that produces an increment in the velocity, 
followed by a braking phase. The flat zones indicate that the asymptotic speed $\vinf$ has 
been reached. The sixth curve shows a case where the friction is too intense and a positive 
acceleration produces a reduction of the speed that settles at $\vinf$; the final condition is then 
met by a rapid braking phase. The last two curves (dashed lines) at the bottom of 
Figure~\ref{fig:test}.1 show two infeasible manoeuvres, which means the friction is 
too high (or equivalently the acceleration $\ap$ is too weak) to reach the final point. 
Also here the flat zones indicate the asymptotic velocity $\vinf$.\\
The second example shows the behaviour of changing the aerodynamic coefficient $c_1$. We keep the 
parameters previously fixed in Table~\ref{tab:2} and consider the values for $c_1$ from $c_1=0$ to 
$c_1=0.03$ as in Table~\ref{tab:2}. The result is showed in Figure \ref{fig:test}.2, where 
the curves are ordered from top to bottom according to increasing values of $c_1$, e.g. the higher
the friction, the slower the velocity.\\
The third experiment makes the value of $\ap$ change from $\ap = 10^{-6}$ to $\ap=10$. The 
other values are kept as in the previous tests. The curves of Figure \ref{fig:test}.3 
are ordered from top to bottom according to the decreasing values of $\ap$. The first 
four are admissible manoeuvres that share a common deceleration curve, the last four 
curves (dashed) are non admissible solutions because the acceleration $\ap$ is not strong 
enough to reach the final position.\\
The fourth test makes the value of $\ab$ change from  $\ab = 10^{-6}$ to $\ab=10$. 
The curves of Figure \ref{fig:test}.4 are ordered from  top to bottom according to the increasing
values of $\ab$, that is, the stronger is the braking, the shorter is the braking phase. The 
curves are admissible manoeuvres that share a common acceleration curve and meet in the final 
point.

\begin{figure}[!tb]
  \begin{center}
    \begin{tabular}{cc}
      \subfigure[Test 1, variation of $c_0$.]{\includegraphics[scale=0.5]{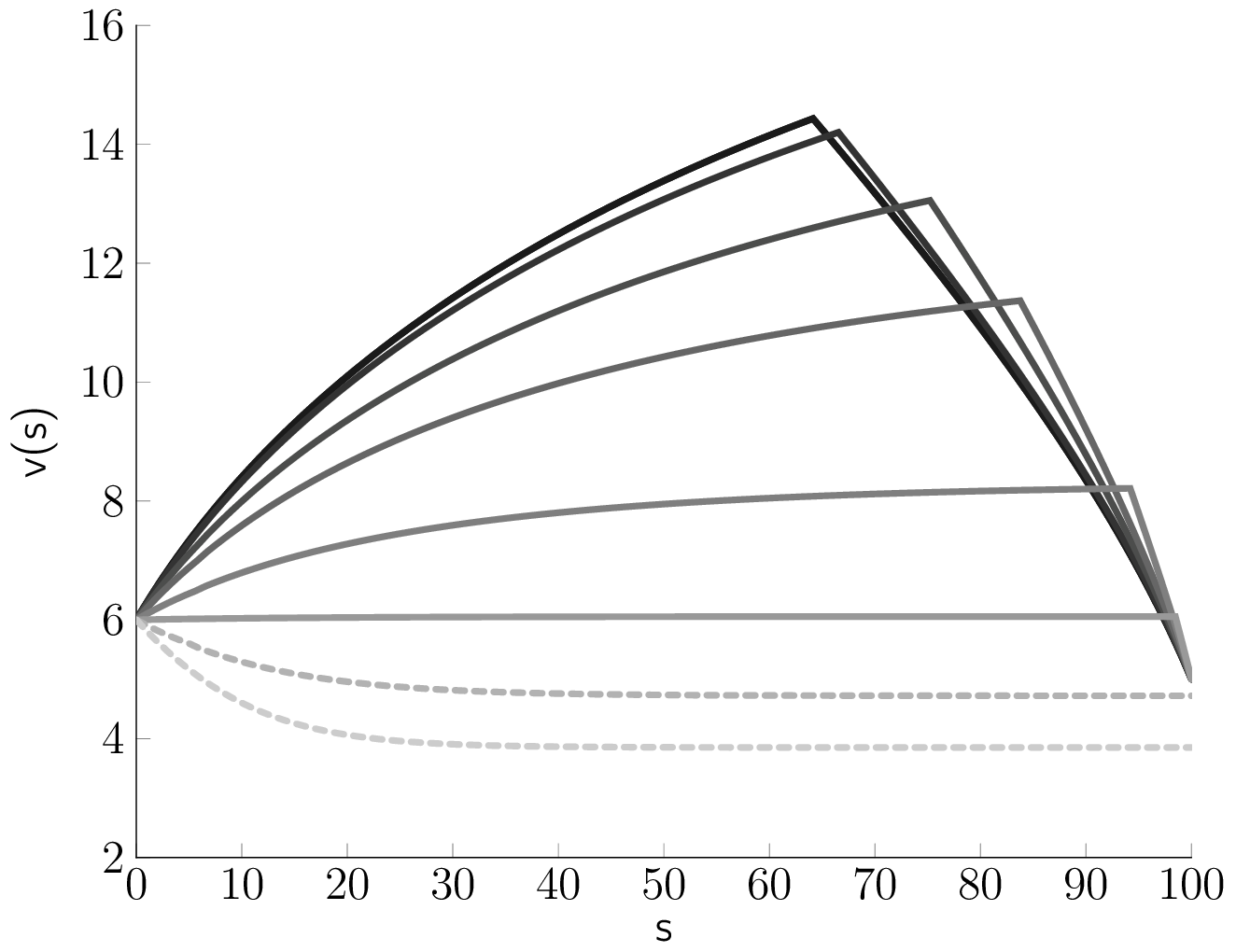}}
      &
      \subfigure[Test 2, variation of $c_1$.]{\includegraphics[scale=0.5]{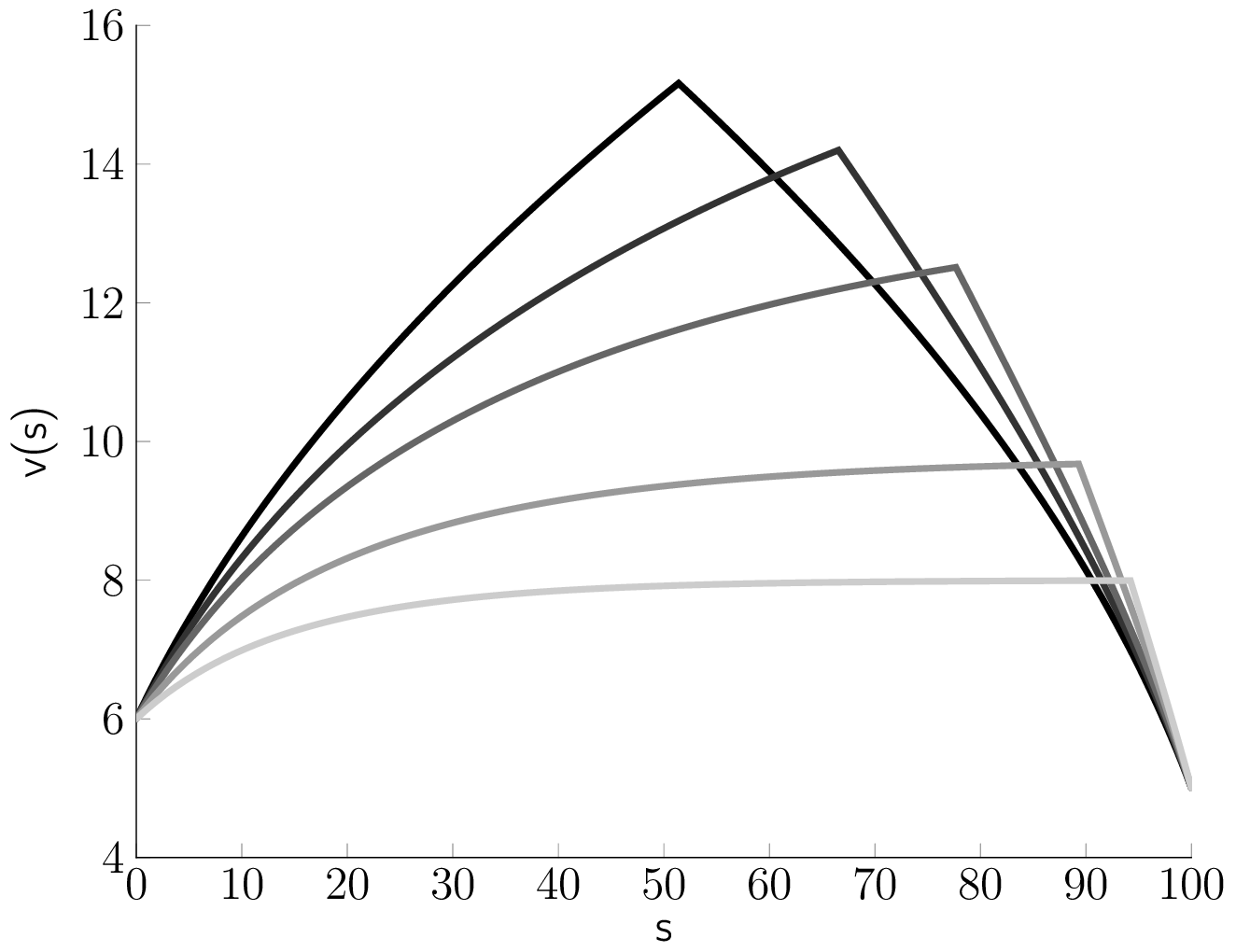}}
      \\
      \subfigure[Test 3, variation of $\ap$.]{\includegraphics[scale=0.5]{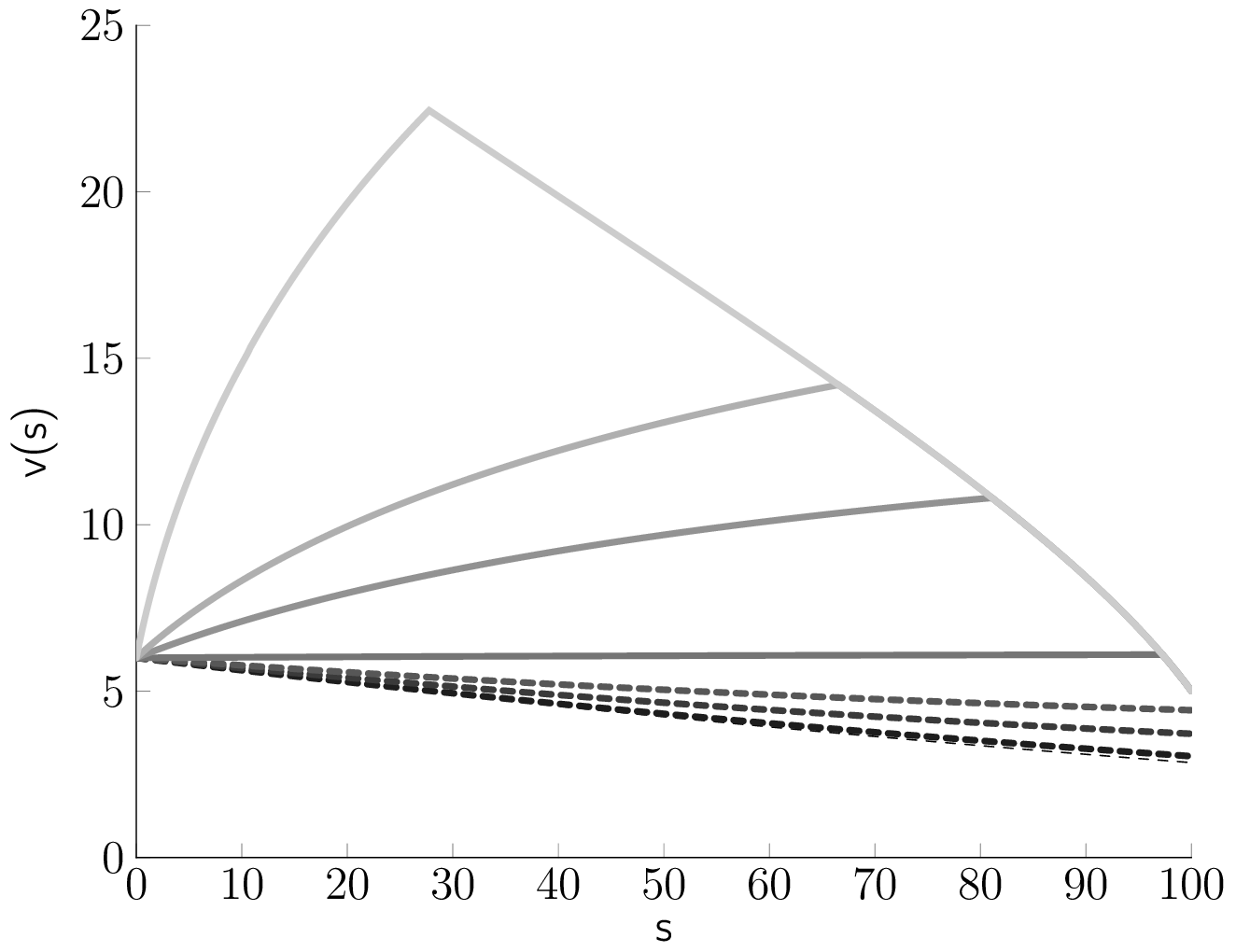}}
      &
      \subfigure[Test 4, variation of $\ab$.]{\includegraphics[scale=0.5]{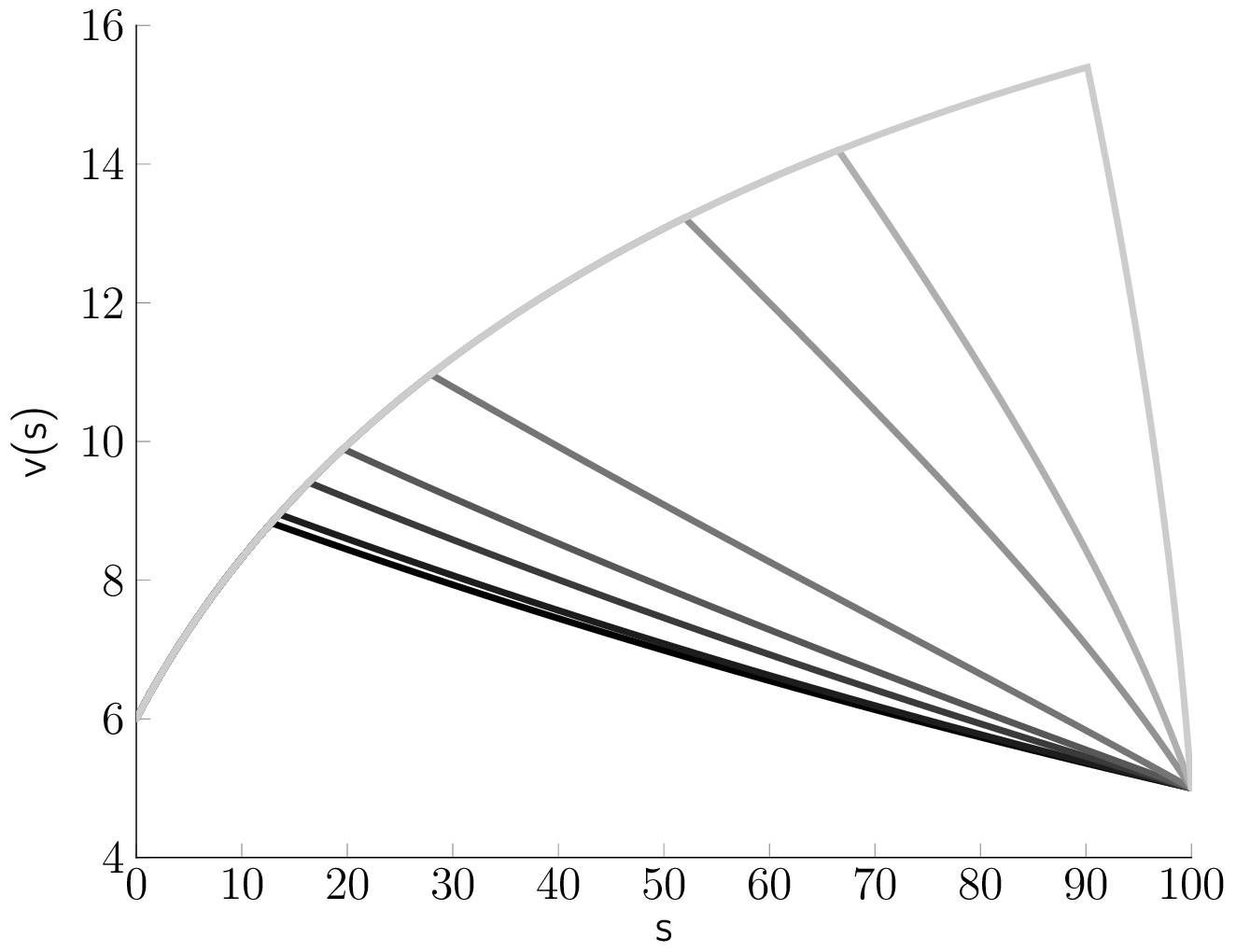}}
      \end{tabular}
   \end{center}
   \caption{The four pictures show the graphical result of the numerical tests.
            Solid lines are feasible manoeuvres while dashed lines are unfeasible manoeuvres.
            Parameters of the computed solutions are taken from Table~\ref{tab:2}}
   \label{fig:test}
   \end{figure}

\section{Conclusions, Future Work and Applications}
In this paper we developed an accurate and fast numerical algorithm for 
the semi-analytic solution for the time optimal control of a car-like model from an initial to a final position and velocity.
The novelty of the OCP is the presence of the quadratic term in the ODE for the velocity, which
yields a nonlinear differential equation of Riccati kind. The nature of the OCP being Bang-Bang
allows the analytic solution of the ODEs and thus the reduction of the problem to the numeric
solution of a curve intersection. An immediate consequence is the avoidance of the complete
numeric integration of the dynamic system and hence a very low computational time. A rough estimate
of the computational cost for the present semi-analytic solution of the OCP is around 
1 microsecond with a C++ implementation, against a pure numerical solution with Pins or GPOPS-II 
of around 1 second.\\
We successfully used it in the solution of the more general optimal control problem of steering 
a car-like model with the additional constraint on the maximum lateral acceleration and 
a given trajectory expressed as clothoid curve, see \cite{ECC:2016}. The effectiveness of
the algorithm presented in this paper was mandatory in the construction of the tool, because
during the optimisation process, many instances of the Bang-Bang algorithm were required, most
of them with parameters that came from other computational processes, yielding non standard 
situations. The na\"ive implementation of the  analytic solution was found to be insufficient
for that purpose, therefore we developed the stable formulas.\\
As a future work based on the present building block, we want to solve the optimal control
herein presented with the limitation on the lateral acceleration and with a trajectory given 
by fixed sequence of clothoids, not just one. Another future study is the design of a tool that gives
the time optimal trajectory on a road, in terms of clothoids, combining the present module with
a planning algorithm. 

\section*{Acknowledgements}
The research activity described in this paper has received
funding from the University of Trento within the project
``OptHySYS - Optimisation techniques for Hybrid dynamical SYStems: from theory to applications''.

\section*{Appendix: Algorithms collected}

\newcommand{\assign}{\leftarrow}
\setlength\columnsep{3pt} 

\begin{algorithm}[H]
  \SetAlgoLined
  \SetKwProg{Fn}{procedure}{}{}
  \KwData{$c_0$,$c_1$,$\vzero$,$\acc$}
  \KwResult{Computation of constants for cases (a)--(e)}
  \eIf{$c_0^2+4\acc c_1\geq 0$}{
    Compute $w$, $\alpha$, $\beta$, $\gamma$, $\vinf$, $\acczero$ with equation~\eqref{eq:constants}\;
    Compute $\tzero$, $\tinf$ with equation~\eqref{eq:t0:def} using expansion~\eqref{eq:aux_functions}\;
    \uIf(\quad\tcp*[h]{case (d)}){$\acc<0$}{
      $\tmax\assign\tzero\assign\Lfun\left(\vzero/(\acc+\vzero\beta),w\right)$\;
      \leIf{$\gamma>w$}{$\tmin\assign\tinf\assign\Lfun\left(\gamma^{-1},w\right)$}
           {$\tmin\assign\tinf\assign-\infty$}
      $\smin\assign-\infty\,;$\quad
      $\smax\assign\Lfun(\acc\Gfun(\tzero)-\vzero\Efun(-\tzero,w)\exp(-\beta\tzero),c_1 )$
      \qquad\tcp*[h]{\textrm{Equation \eqref{eq:lem6:1:b} with $t=\tzero$}}
    }
    \uElseIf(\quad\tcp*[h]{case (c)}){$\vzero<\vinf$}{
      $\smin\assign-\infty;$\quad$\smax\assign\infty;$\quad$\tmax\assign\infty$\;
      \leIf{$\gamma>w$}{$\tmin\assign\tinf\assign\Lfun\left(\gamma^{-1},w\right)$}
           {$\tmin\assign\tinf\assign-\infty$}
    }
    \uElseIf(\quad\tcp*[h]{case (b)}){$\vzero>\vinf$}{
      $\tmin\assign\tzero\assign\Lfun\left(\vzero/(\acc+\vzero\beta),w\right)$\quad
      $\tmax\assign\infty;$\quad$\smax\assign\infty$\;
      \lIf
      {$w>0.1$}{$\smin\assign (\Lfun(\vzero,\alpha/\acc) + \Lfun(-\vzero, c_1/\alpha ) )/ w$
      \qquad \tcp*[h]{Equation \eqref{eq:s0sv0} with $t=\tzero$}}
      \lElse{$\smin\assign\Lfun(\acc\Gfun(\tzero)-\vzero\Efun(-\tzero,w)\exp(-\beta\tzero),c_1 )$
      \qquad\tcp*[h]{Equation \eqref{eq:lem6:1:b} with $t=\tzero$}}
    }
    \Else(\quad\tcp*[h]{case (a)}){
      $\tmin\assign-\infty;$\quad
      $\tmax\assign\infty;$\quad
      $\smin\assign-\infty;$\quad
      $\smax\assign\infty$\;
    }
  }(\tcp*[h]{case (e)}){
    Compute $\abs{w}$, $\theta$, $\theta_0$, $\theta_1$, $\acczero$, $\ell_1$
    with equation~\eqref{eq:constants} and \eqref{eq:def:ell1}
    and $\tvzero$ with equation~\eqref{eq:t0:def}\;
  }
  \caption{Setup constants for forward computation of $v(t;c_0,c_1,\vzero,\acc)$ and  $s(t;c_0,c_1,\vzero,\acc)$}
  \label{algo:setup}
\end{algorithm}

\begin{algorithm}[H]
  \SetAlgoLined
  \SetKwFunction{setup}{setup}
  \SetKwProg{Fn}{procedure}{}{}
  \KwData{$t$ and constant computed with Algorithm~\ref{algo:setup}}
  \KwResult{Computation of $v(t)$}
  \Switch{ which computational case? }{
  \Case{case (a)-(b)-(c)-(d)}{
    \uIf{ $t \leq  0$ \textbf{or} $t > 1$}{
       $q\assign 1 + (\gamma-w)\Efun(-t,w)$\;
       \lIf{ $\abs{t-\tzero} < 100\epsilon$}{ 
          $p\assign (\alpha\vzero-\acc)\Efun(t-\tzero, w );\;$ \Return $p/q$
       }
       \lElse{$p\assign (\acc-\acczero)\Efun(-t,w);\;$\Return $\vzero+p/q$}
    }
    \Else{
      \leIf{$\abs{t-\tinf} < 100\epsilon$}{$q\assign(w-\gamma)\Efun(t-\tinf,w)$}
           {$q\assign 1 - \gamma\Efun(t,w)$}
      \lIf{ $\abs{t-\tzero} < 100\epsilon$}
        {$p\assign(\alpha\vzero-\acc)\Efun(t-\tzero,w);\;$ \Return $p/q$}
      \lElse{$p\assign(\acczero-\acc)\Efun(t,w);\;$\Return $\vzero+p/q$}
     }
   }
   \Case{case (e)}{
      $c\assign \cos(t\abs{w}/2);$\quad
      $s\assign \Sfun(t/2,\abs{w})$\quad
      $p\assign \vzero c-(c_0\vzero-2\acc)s;$\quad
      $q\assign c+(2c_1\vzero+c_0)s$\;
      \Return{$p/q$}
    }
  }
  \caption{Compute $v(t)\DEF v(t;c_0,c_1,\vzero,\acc)$}
  \label{algo:velocity}
\end{algorithm}

\begin{algorithm}[H]
  \SetAlgoLined
  \SetKwFunction{setup}{setup}
  \SetKwProg{Fn}{procedure}{}{}
  \KwData{$t$ and constant computed with \setup($c_0$,$c_1$,$\vzero$,$\acc$)}
  \KwResult{Computation of $s(t)$}
  \Switch{ which computational case? }{
  \Case{case (a)-(b)-(c)-(d)}{
    \lIf{ $t\geq\tmax$}{\Return $\smax$}
    \lElseIf{ $t \leq  \tmin$}{\Return $\smin$}
    \lElse{$\Return\Lfun(\acc\Gfun(t)-\vzero\Efun(-t,w)\exp(-\beta t),c_1 )$}
  }
  \Case{case (e)}{
    \eIf{ $\abs{ t\ell_1}\leq 0.001$ }{
      \Return $((\abs{\acc}+\acczero)\Qfun( t\ell_1,\cos(\theta_1))t+\vzero)t$
    }{
      $\arg\assign (2c_1\vzero+c_0)\Sfun(t/2,\abs{w})+\cos(t\abs{w}/2)$\;
      \leIf{ $arg\leq 0$}{\Return $-\infty$}
           {\Return $(\log(\arg) - c_0t/2)/c_1$}
    }
  }
  }
  \caption{Compute $s(t)\DEF s(t;c_0,c_1,\vzero,\acc)$}
  \label{algo:space}
\end{algorithm}

\begin{algorithm}[H]
  \SetAlgoLined
  \SetKwProg{Fn}{procedure}{}{}
  \KwData{$\tguess$}
  \KwResult{Computation of $t$ such that $f(t)=0$}
  $t\assign\tguess$\;   
  \Repeat{$\abs{\delta t}<\epsilon$}{
    \leIf{ $f(t)>0$}{
      $\delta t\assign f(t)\big/\big(f'(t)+f(t)/(t-\tmin)\big)$ 
    }{
      $\delta t\assign f(t)/f'(t)$  
    }
    $t\assign t-\delta t$ \;   
  }
  \Return $t$\;
  \caption{Modified Newton for the computation $f(t)=0$ when $\tmin>-\infty$}
  \label{algo:newton}
\end{algorithm}

\bibliographystyle{spmpsci}
\bibliography{OCP_pure_BangBang-biblio}

\begin{thebibliography}{10}
\providecommand{\url}[1]{{#1}}
\providecommand{\urlprefix}{URL }
\expandafter\ifx\csname urlstyle\endcsname\relax
  \providecommand{\doi}[1]{DOI~\discretionary{}{}{}#1}\else
  \providecommand{\doi}{DOI~\discretionary{}{}{}\begingroup
  \urlstyle{rm}\Url}\fi

\bibitem{abramowitz:1964}
Abramowitz, M., Stegun, I.: Handbook of Mathematical Functions with Formulas,
  Graphs, and Mathematical Tables.
\newblock No.~55 in National Bureau of Standards Applied Mathematics Series.
  U.S. Government Printing Office, Washington, D.C. (1964)

\bibitem{Amat:2008}
Amat, S., Hern{\'a}ndez, M., Romero, N.: A modified chebyshev's iterative
  method with at least sixth order of convergence.
\newblock Applied Mathematics and Computation \textbf{206}(1), 164 -- 174
  (2008).
\newblock \doi{http://dx.doi.org/10.1016/j.amc.2008.08.050}.
\newblock
  \urlprefix\url{http://www.sciencedirect.com/science/article/pii/S0096300308006449}

\bibitem{Amditis:2010fx}
Amditis, A., Bertolazzi, E., Bimpas, M., Biral, F., Bosetti, P., Da~Lio, M.,
  Danielsson, L., Gallione, A., Lind, H., Saroldi, A., Sj{\"o}gren, A.: A
  holistic approach to the integration of safety applications: The insafes
  subproject within the european framework programme 6 integrating project
  prevent.
\newblock IEEE Transactions on Intelligent Transportation Systems
  \textbf{11}(3), 554--566 (2010).
\newblock \doi{10.1109/TITS.2009.2036736}

\bibitem{Ascher:1995}
Ascher, U., Mattheij, R., Russell, R.: Numerical Solution of Boundary Value
  Problems for Ordinary Differential Equations.
\newblock Society for Industrial and Applied Mathematics (1995).
\newblock \doi{10.1137/1.9781611971231}

\bibitem{Athans:1966}
Athans, M., Falb, P.: Optimal Control: An Introduction to the Theory and Its
  Applications.
\newblock Lincoln Laboratory publications. McGraw-Hill (1966)

\bibitem{Bertolazzi:2006kc}
Bertolazzi, E., Biral, F., Da~Lio, M.: Symbolic-numeric efficient solution of
  optimal control problems for multibody systems.
\newblock Journal of Computational and Applied Mathematics \textbf{185}(2),
  404--421 (2006).
\newblock \doi{10.1016/j.cam.2005.03.019}

\bibitem{Bertolazzi:2010ud}
Bertolazzi, E., Biral, F., Da~Lio, M., Saroldi, A., Tango, F.: Supporting
  drivers in keeping safe speed and safe distance: The saspence subproject
  within the european framework programme 6 integrating project prevent.
\newblock IEEE Transactions on Intelligent Transportation Systems
  \textbf{11}(3), 525--538 (2010).
\newblock \doi{10.1109/TITS.2009.2035925}

\bibitem{Bertolazzi:2015aa}
Bertolazzi, E., Frego, M.: G1 fitting with clothoids.
\newblock Mathematical Methods in the Applied Sciences \textbf{38}(5), 881--897
  (2015).
\newblock \doi{10.1002/mma.3114}

\bibitem{Betts:2010}
Betts, J.: Practical Methods for Optimal Control and Estimation Using Nonlinear
  Programming, second edn.
\newblock Society for Industrial and Applied Mathematics (2010).
\newblock \doi{10.1137/1.9780898718577}

\bibitem{Betts:2007}
Betts, J.T., Campbell, S.L., Engelsone, A.: Direct transcription solution of
  optimal control problems with higher order state constraints: theory vs
  practice.
\newblock Optimization and Engineering \textbf{8}(1), 1--19 (2007).
\newblock \doi{10.1007/s11081-007-9000-8}

\bibitem{Biral:2014aa}
Biral, F., Bertolazzi, E., {Da Lio}, M.: The Optimal Manoeuvre, chap.~5, pp.
  119--154.
\newblock John Wiley \& Sons, Ltd (2014).
\newblock \doi{10.1002/9781118536391.ch5}

\bibitem{Buskens:2013aa}
B{\"u}skens, C., Wassel, D.: The {ESA} {NLP} solver {WORHP}.
\newblock In: G.~Fasano, J.D. Pint{\'e}r (eds.) Modeling and Optimization in
  Space Engineering, \emph{Springer Optimization and Its Applications},
  vol.~73, pp. 85--110. Springer New York (2013).
\newblock \doi{10.1007/978-1-4614-4469-5_4}

\bibitem{Campbell:2016}
Campbell, S.L., Betts, J.T.: Comments on direct transcription solution of dae
  constrained optimal control problems with two discretization approaches.
\newblock Numerical Algorithms pp. 1--32 (2016).
\newblock \doi{10.1007/s11075-016-0119-6}

\bibitem{Chun:2009}
Chun, C., Neta, B.: Certain improvements of newton's method with fourth-order
  convergence.
\newblock Applied Mathematics and Computation \textbf{215}(2), 821 -- 828
  (2009).
\newblock \doi{10.1016/j.amc.2009.06.007}

\bibitem{Conte:1980}
Conte, S.D., Boor, C.W.D.: Elementary Numerical Analysis: An Algorithmic
  Approach, 3rd edn.
\newblock McGraw-Hill Higher Education (1980)

\bibitem{dubins1957curves}
Dubins, L.E.: On curves of minimal length with a constraint on average
  curvature, and with prescribed initial and terminal positions and tangents.
\newblock American Journal of mathematics pp. 497--516 (1957)

\bibitem{ECC:2016}
Frego, M., Bertolazzi, E., Biral, F., Fontanelli, D., Palopoli, L.:
  Semi-analytical minimum time solutions for a vehicle following clothoid-based
  trajectory subject to velocity constraints.
\newblock In: European Control Conference (ECC), p.~7 (2016)

\bibitem{FregoBBFP16Automatica}
Frego, M., Bertolazzi, E., Biral, F., Fontanelli, D., Palopoli, L.:
  {Semi-Analytical Minimum Time Solutions with Velocity Constrainst for
  Trajectory Following of Vehicles}.
\newblock Automatica  (2016).
\newblock Submitted

\bibitem{FregoBBFP16CDC}
Frego, M., Bevilacqua, P., Bertolazzi, E., Biral, F., Fontanelli, D., Palopoli,
  L.: {Trajectory Planning for car--like vehicles: a modular approach}.
\newblock Conference on Decision and Control (CDC 2016)  (2016).
\newblock Submitted

\bibitem{Gerdts:2013b}
Gerdts, M., Xausa, I.: Avoidance trajectories using reachable sets and
  parametric sensitivity analysis.
\newblock IFIP Advances in Information and Communication Technology \textbf{391
  AICT}, 491--500 (2013).
\newblock \doi{10.1007/978-3-642-36062-6_49}

\bibitem{Diehl:2011}
Houska, B., Ferreau, H., Diehl, M.: {ACADO} {T}oolkit -- {A}n {O}pen {S}ource
  {F}ramework for {A}utomatic {C}ontrol and {D}ynamic {O}ptimization.
\newblock Optimal Control Applications and Methods \textbf{32}(3), 298--312
  (2011)

\bibitem{Gerdts:2013}
Landry, C., Gerdts, M., Henrion, R., H{\"o}mberg, D.: Path-planning with
  collision avoidance in automotive industry.
\newblock IFIP Advances in Information and Communication Technology \textbf{391
  AICT}, 102--111 (2013).
\newblock \doi{10.1007/978-3-642-36062-6_11}

\bibitem{DaLio2014}
Lio, M.D., Biral, F., Bertolazzi, E., Galvani, M., Bosetti, P., Windridge, D.,
  Saroldi, A., Tango, F.: Artificial co-drivers as a universal enabling
  technology for future intelligent vehicles and transportation systems.
\newblock IEEE Transactions on Intelligent Transportation Systems
  \textbf{16}(1), 244--263 (2015).
\newblock \doi{10.1109/TITS.2014.2330199}

\bibitem{Boyd:2014}
Lipp, T., Boyd, S.: Minimum-time speed optimisation over a fixed path.
\newblock International Journal of Control \textbf{87}(6), 1297--1311 (2014).
\newblock \doi{10.1080/00207179.2013.875224}

\bibitem{Lotfi:2014}
Lotfi, T., Sharifi, S., Salimi, M., Siegmund, S.: A new class of three-point
  methods with optimal convergence order eight and its dynamics.
\newblock Numerical Algorithms \textbf{68}(2), 261--288 (2014).
\newblock \doi{10.1007/s11075-014-9843-y}

\bibitem{Mazzia:2002}
Mazzia, F., Sgura, I.: Numerical approximation of nonlinear {BVPs} by means of
  {BVMs}.
\newblock Applied Numerical Mathematics \textbf{42}(1--3), 337 -- 352 (2002).
\newblock \doi{http://dx.doi.org/10.1016/S0168-9274(01)00159-3}

\bibitem{mazzia:2011}
Mazzia, F., Trigiante, D.: Numerical solution of optimal control problems.
\newblock In: S.~Sivasundaram (ed.) Advances in Dynamics and Control: Theory
  Methods and Applications, Mathematical problems in engineering, aerospace and
  sciences. Cambridge Scientific Publ. (2011)

\bibitem{Neta:2013}
Neta, B., Scott, M.: On a family of halley-like methods to find simple roots of
  nonlinear equations.
\newblock Applied Mathematics and Computation \textbf{219}(15), 7940 -- 7944
  (2013).
\newblock \doi{10.1016/j.amc.2013.02.035}

\bibitem{nocedal:2006}
Nocedal, J., Wright, S.: Numerical Optimization.
\newblock Springer Series in Operations Research and Financial Engineering.
  Springer New York (2006)

\bibitem{nist:2010}
Olver, F.J., Lozier, D., Boisvert, R., Clark, C.: N{IST} handbook of
  mathematical functions.
\newblock U.S. Department of Commerce National Institute of Standards and
  Technology, Washington, DC (2010).
\newblock With CD-ROM

\bibitem{Rao:2014}
Patterson, M.A., Rao, A.V.: {GPOPS-II}: A {MATLAB} software for solving
  multiple-phase optimal control problems using hp-adaptive gaussian quadrature
  collocation methods and sparse nonlinear programming.
\newblock ACM Trans. Math. Softw. \textbf{41}(1), 1:1--1:37 (2014).
\newblock \doi{10.1145/2558904}

\bibitem{Neta:2014}
Petkovi{\'c}, M.S., Neta, B., Petkovi{\'c}, L.D., D{\v z}uni{\'c}, J.:
  Multipoint methods for solving nonlinear equations: A survey.
\newblock Applied Mathematics and Computation \textbf{226}, 635 -- 660 (2014).
\newblock \doi{10.1016/j.amc.2013.10.072}

\bibitem{Rao:2014b}
Rao, A.: Trajectory optimization: A survey.
\newblock Lecture Notes in Control and Information Sciences \textbf{455 LNCIS},
  3--21 (2014).
\newblock \doi{10.1007/978-3-319-05371-4_1}

\bibitem{rizano2013global}
Rizano, T., Fontanelli, D., Palopoli, L., Pallottino, L., Salaris, P.: Global
  path planning for competitive robotic cars.
\newblock In: Decision and Control (CDC), 2013 IEEE 52nd Annual Conference on,
  pp. 4510--4516. IEEE (2013)

\bibitem{Frazzoli:2014}
Sanfelice, R., Yong, S., Frazzoli, E.: On minimum-time paths of bounded
  curvature with position-dependent constraints.
\newblock Automatica \textbf{50}(2), 537--546 (2014).
\newblock
  \urlprefix\url{http://ares.lids.mit.edu/papers/Sanfelice.Yong.ea.AUT13.pdf}

\bibitem{Velenis:2005aa}
Velenis, E., Tsiotras, P.: Optimal velocity profile generation for given
  acceleration limits: theoretical analysis.
\newblock In: American Control Conference, 2005. Proceedings of the 2005, pp.
  1478--1483 vol. 2 (2005).
\newblock \doi{10.1109/ACC.2005.1470174}

\bibitem{Velenis:2008}
Velenis, E., Tsiotras, P.: Minimum-time travel for a vehicle with acceleration
  limits: Theoretical analysis and receding-horizon implementation.
\newblock Journal of Optimization Theory and Applications \textbf{138}(2),
  275--296 (2008).
\newblock \doi{10.1007/s10957-008-9381-7}

\bibitem{Diehl:2009}
Verscheure, D., Demeulenaere, B., Swevers, J., Schutter, J.D., Diehl, M.:
  Time-optimal path tracking for robots: A convex optimization approach.
\newblock IEEE Transactions on Automatic Control \textbf{54}(10), 2318--2327
  (2009).
\newblock \doi{10.1109/TAC.2009.2028959}

\bibitem{Biegler:2005}
W\"achter, A., Biegler, T.L.: On the implementation of an interior-point filter
  line-search algorithm for large-scale nonlinear programming.
\newblock Mathematical Programming \textbf{106}(1), 25--57 (2005).
\newblock \doi{10.1007/s10107-004-0559-y}

\bibitem{ICLOCS:2010}
Wyk, E.J.V., Falugi, P., Kerrigan, E.: {ICLOCS} (2010).
\newblock \urlprefix\url{http://www.ee.ic.ac.uk/ICLOCS}

\bibitem{zaitsev:2002}
Zaitsev, V., Polyanin, A.: Handbook of Exact Solutions for Ordinary
  Differential Equations.
\newblock CRC Press (2002)

\end{thebibliography}

\end{document}